\documentclass[a4paper]{article}

\usepackage[english]{babel}
\usepackage[utf8]{inputenc} 
\usepackage[T1]{fontenc}

\usepackage[a4paper,left=3cm,right=2.3cm,top=2.8cm,bottom=4cm]{geometry}

\usepackage{times}
\usepackage{hyperref}
\usepackage{url}
\usepackage{amsmath}
\usepackage{amssymb}
\usepackage{amsthm}
\usepackage{mathtools,enumitem}
\usepackage{algorithm}
\usepackage{algorithmic}
\usepackage{subcaption}
\usepackage{graphicx}
\usepackage{rotating}

\usepackage{authblk}
\usepackage{tikz,upgreek,adjustbox}

\usepackage{pgfplots}
\usepackage{pgfkeys}
\usetikzlibrary{spy}
\usetikzlibrary[patterns] 
\usetikzlibrary{calc,decorations.pathreplacing} 
\usetikzlibrary{shadows.blur} 
\usetikzlibrary{shapes.symbols}
\newcommand{\scal}[2]{\left\langle #1 , #2 \right\rangle}
\DeclareMathOperator{\IR}{\mathbb{R}}
\DeclareMathOperator*{\argmin}{argmin}
\DeclareMathOperator*{\argmax}{argmax}

\DeclareMathOperator{\Ccal}{\mathcal{C}}

\DeclareMathOperator{\diag}{diag}
\DeclareMathOperator{\KL}{KL}
\DeclareMathOperator{\SK}{\textrm{SK}}
\newcommand{\norm}[1]{\left\lVert #1 \right\rVert}
\renewcommand{\epsilon}{\varepsilon}
\newcommand{\func}{\Theta}

\newcommand{\da}{\alpha} 
\newcommand{\db}{\beta} 
\newcommand{\uu}{u}  
\newcommand{\vv}{v}  
\newcommand{\ww}{s} 

\theoremstyle{plain}
\newtheorem{theorem}{Theorem}
\newtheorem{proposition}{Proposition}
\newtheorem{lemma}{Lemma}
\newtheorem{corollary}{Corollary}

\theoremstyle{definition}

\theoremstyle{remark}
\newtheorem{remark}{Remark}

\begin{document}

\title{Overrelaxed Sinkhorn--Knopp Algorithm for Regularized Optimal Transport}

\author[1]{Alexis Thibault}
\author[2]{L\'ena\"ic Chizat}
\author[3]{Charles Dossal}
\author[4]{Nicolas Papadakis}
\date{}      
%
\affil[1]{Inria Bordeaux Sud Ouest, Talence, France}
\affil[2]{CNRS, Laboratoire de Mathématiques d'Orsay, Université Paris-Saclay, France}

\affil[3]{INSA Toulouse,    France}
 
\affil[4]{CNRS, Univ. Bordeaux, IMB, F-33400 Talence, France}

\date{}

\maketitle
\abstract{This article describes a set of methods for quickly computing the solution to the regularized optimal transport problem. It generalizes and improves upon the widely-used iterative Bregman projections algorithm (or Sinkhorn--Knopp algorithm).
We first propose to rely on regularized nonlinear acceleration schemes.  In practice, such approaches lead to  fast algorithms, but their global convergence is not ensured.
Hence, we next propose a new algorithm with convergence guarantees.
The idea is to overrelax the Bregman projection operators, allowing for faster convergence. 
We propose a simple method for establishing global convergence by ensuring the decrease of a Lyapunov function at each step.
An adaptive choice of overrelaxation parameter based on the Lyapunov function is constructed.
We also suggest a heuristic to choose a suitable asymptotic overrelaxation parameter, based on a local convergence analysis. Our numerical experiments show a gain in convergence speed by an order of magnitude in certain regimes.}

\section{Introduction}
Optimal Transport is an efficient and flexible tool to compare two probability distributions which has been popularized in the computer vision community in the context of discrete histograms \cite{Rubner2000}. The introduction of entropic regularization of the optimal transport problem in \cite{cuturi13} has made possible the use of the fast Sinkhorn--Knopp algorithm \cite{sinkhorn64}   scaling with high dimensional data. 
Regularized optimal transport have thus been intensively used  in  Machine Learning with applications such as   Geodesic PCA \cite{seguy2015principal}, domain adaptation \cite{2015arXiv150700504C}, data fitting \cite{2015arXiv150605439F},  training of Boltzmann Machine \cite{NIPS2016_6248}  or dictionary learning \cite{Rolet2016,2017arXiv170801955S}.

The computation of optimal transport between two data relies on the estimation of an optimal transport matrix, the entries of which represent the quantity of mass transported between  data locations. 
Regularization of optimal transport with strictly convex regularization \cite{cuturi13, dessein2016}  nevertheless involves a spreading of the mass. Hence, for particular purposes such as color interpolation \cite{Rabin2014} or gradient flow \cite{2016arXiv160705816C}, it is  necessary  to consider small parameters $\varepsilon$ for the entropic regularization term. The Sinkhorn-Knopp (SK) algorithm is a state of the art algorithm to solve the regularized transport problem. The SK algorithm performs alternated projections and the sequence of generated iterates converges to a solution of the regularized transport problem. Unfortunatly, the lower $\varepsilon$ is, the slower the SK algorithm converges. To improve the convergence rate of SK, several acceleration strategies have been proposed in the literature, based for example on mixing or over-relaxation. 


\subsection{Accelerations of the Sinkhorn--Knopp Algorithm.}
In the literature, several accelerations of the Sinkhorn--Knopp algorithm have been proposed, using for instance greedy coordinate descent \cite{altschuler2017near} or screening strategies \cite{alaya2019screening}. In another line of research, the introduction of relaxation variables through heavy ball approaches \cite{POLYAK19641} has recently gained in popularity  to speed up the convergence of algorithms optimizing convex \cite{2014arXiv1412.7457G} or non convex \cite{Zavriev1993,2016arXiv160609070O} problems. 
In this context, the use of regularized nonlinear accelerations (RNA)  \cite{anderson1965iterative,scieur2018online} based on the Anderson mixing have reported important numerical improvements, although the  global convergence  is not guaranteed with such approaches as it is shown further. In this paper we also investigate another approach related to the successive overrelexation (SOR) algorithm~\cite{young2014iterative}, which is a classical way to solve linear systems. Similar schemes have been empirically considered to accelerate the SK algorithm  in \cite{peyre2016quantum,2017arXiv170801955S}. The convergence of these algorithms has nevertheless not been studied yet in the context of regularized optimal transport.

\subsection{Overview and contributions}
The contribution of this paper is twofold. First, the numerical efficiency of the RNA methods applied to the SK algorithm to solve the regularized transport problem is shown. Second, a new  extrapolation and relaxation technique for accelerating the Sinkhorn--Knopp (SK) algorithm, ensuring convergence is given. The numerical efficiency of this new algorithm is demonstrated and an heuristic rule is also proposed to improve the rate of the algorithm.

Section 2 is  devoted to the Sinkhorn--Knopp algorithm.  In Section 3, we propose to apply regularized nonlinear acceleration (RNA) schemes to the SK algorithm. We experimentally show that such methods lead to impressive accelerations for low values of the entropic regularization parameter. 
In order to have a globally converging method, we then propose a new  overrelaxed algorithm. In Section 4, we show the global convergence of our algorithm and analyse its local convergence rate  to justify the acceleration.
We finally demonstrate numerically  in Section 5 the interest of our method. Larger accelerations are indeed observed for decreasing values of the entropic regularization parameter.

\begin{remark} This paper is an updated version of an unpublished work \cite{thibault2017overrelaxed} presented at the NIPS 2017 Workshop on  Optimal Transport \& Machine Learning. In the meanwhile, complementary results on the global convergence of our method presented in Section 4 have been provided in \cite{lehmann2020note}. The authors show the existence of a parameter ${\theta_0}$ such that   both global convergence and local acceleration are ensured for overrelaxation parameters $\omega\in (1,{\theta_0})$. This result is nevertheless theoretical and the numerical estimation of ${\theta_0}$ is still an open question. With respect to our unpublished work \cite{thibault2017overrelaxed}, the current article presents an original contribution  in section \ref{sec:RNA}: the application of RNA methods to accelerate the convergence of the SK algorithm.
\end{remark}

\section{Sinkhorn algorithm}
Before going into further details, we now briefly introduce the main notations and concepts used all along this article.

\subsection{Discrete optimal transport}
We consider two discrete probability measures $\mu_k \in \IR_{+*}^{n_k}$.
Let us define the two following linear operators
\begin{align*}
	A_1 &: \begin{cases}
		\IR^{n_1 n_2} \rightarrow \IR^{n_1} \\
		(A_1 x)_i = \sum_j x_{i,j}
	\end{cases} &
	A_2 &: \begin{cases}
		\IR^{n_1 n_2} \rightarrow \IR^{n_2}\\
		(A_2 x)_j = \sum_i x_{i,j},
	\end{cases}
\end{align*}
as well as the affine constraint sets
\begin{align*}
	\Ccal_k &= \left\{ \gamma\in\IR^{n_1 n_2} \mid A_k \gamma = \mu_k \right\}.
\end{align*}
Given a cost matrix $c$ with nonnegative coefficients, where $c_{i,j}$ represents the cost of moving mass $(\mu_1)_i$ to $(\mu_2)_j$,  the optimal transport problem corresponds to the estimation of an optimal transport matrix $\gamma$ solution of:
$$\min_{\gamma\in\Ccal_1\cap \Ccal_2\cap \IR^{n_1 n_2}_+} \langle c,\gamma\rangle:=\sum_{i,j}c_{i,j}\gamma_{i,j}.$$
This is a linear programming problem whose resolution becomes intractable for large problems.

\subsection{Regularized optimal transport}

In \cite{cuturi13}, it has been proposed to regularize this problem by adding a strictly convex entropy regularization:
\begin{equation}\label{ROT}
	\min_{\gamma\in\Ccal_1\cap \Ccal_2\cap \IR^{n_1 n_2}_{+}}K^\epsilon(\gamma) \coloneqq \scal{c}{\gamma} 
	+ \epsilon \KL(\gamma,\mathbf{1})
	,\end{equation}
with $\epsilon>0$, $\mathbf{1}$ is the matrix of size $n_1\times n_2$ full of ones and the Kullback-Leibler divergence is
\begin{equation}\label{KL}
	\KL(\gamma,\zeta) = \sum_{i,j} \gamma_{i,j} \left( \log \left( \frac{\gamma_{i,j}}{\zeta_{i,j}} \right) -1  \right) + \sum_{i,j} \zeta_{i,j}
\end{equation}
with the convention $0 \log 0= 0$. It was shown in \cite{benamou15}  that the regularized optimal transport matrix $\gamma^*$, which is the unique minimizer of problem \eqref{ROT},  is the Bregman projection of $\gamma^0 = e^{-c/\epsilon}$ (here and in the sequel, exponentiation is meant entry-wise) onto $\Ccal_1 \cap \Ccal_2$:
\begin{equation}\label{eq:reg_ot_pb}
	\gamma^* = \argmin_{\Ccal_1 \cap \Ccal_2} K^\epsilon(\gamma)= P_{\Ccal_1 \cap \Ccal_2} (e^{-c/\epsilon}),
\end{equation}
where $P_{\Ccal}$ is the  Bregman projection onto $\Ccal$ defined as
\[
P_{\Ccal}(\zeta) \coloneqq \argmin_{\gamma \in \Ccal} \KL(\gamma,\zeta).
\]

\subsection{Sinkhorn--Knopp algorithm}\label{sec:SK}
Iterative Bregman projections onto $\Ccal_1$ and $\Ccal_2$ converge to a point in the intersection $\Ccal_1 \cap \Ccal_2$ \cite{bregman67}. Hence, the so-called Sinkhorn--Knopp algorithm (SK) \cite{sinkhorn64} that performs alternate Bregman projections, can be considered to compute the regularized  transport matrix:
\begin{align*}
	\gamma^0 &= e^{-c/\epsilon} &
	\gamma^{\ell+1} = P_{\Ccal_2}(P_{\Ccal_1}(\gamma^{\ell})),
\end{align*}
and we have 
$\lim_{l\rightarrow +\infty} \gamma^{\ell} = P_{\Ccal_1 \cap \Ccal_2}(\gamma^0) = \gamma^*.$

In the discrete setting, these projections correspond to diagonal scalings of the input:
\begin{align}\label{scaling}
	P_{\Ccal_1}(\gamma) &= \diag(\uu) \gamma &\text{with}\quad
	\uu &=  {\mu_1}\oslash{A_1 \gamma} \\
	P_{\Ccal_2}(\gamma) &= \gamma \diag(\vv) &\text{with}\quad
	\vv &= {\mu_2}\oslash{A_2 \gamma}\nonumber
\end{align}
where $\oslash$ is the pointwise division. 
To compute numerically the solution one simply has to store $(\uu^{\ell}, \vv^{\ell})\in\IR^{n_1}\times \IR^{n_2}$ and to iterate
\begin{align}\label{eq:update_SK}
	\uu^{\ell+1} &= {\mu_1}\oslash{\gamma^0 \vv^{\ell}} &
	\vv^{\ell+1} &= {\mu_2}\oslash{^t \gamma^0 \uu^{\ell+1}} .
\end{align}
We then have $\gamma^{\ell} = \diag(\uu^{\ell}) \gamma^0 \diag(\vv^{\ell}).$ 

Another way to interpret the SK algorithm is as an alternate maximization algorithm on the dual of the regularized optimal transport problem, see \cite{CuturiPeyre19}, Remark 4.24. The dual problem of \eqref{ROT} is
\begin{equation}\label{DROT}
	\max_{\substack{\da\in \IR^n\\\db\in \IR^m}}\; E(\da,\db) \coloneqq \langle \da,\mu_1\rangle+\langle \db,\mu_2\rangle-\epsilon\sum_{i,j}e^{(\da_i+\db_j-c_{i,j})/\epsilon}.
\end{equation}
As the function $E$ is concave, continuously differentiable and admits a maximizer, the following alternate maximization algorithm converges to a global optimum:

\begin{align}\label{algo_SK}
	\da^{\ell+1} &=  \argmax_\da E(\da,\db^{\ell})\\
	\db^{\ell + 1}&=\argmax_\db E(\da^{\ell+1},\db).\label{algo_SKb}
\end{align}
The explicit solutions of previous problems write
\begin{align}\label{algo_SK2}
	\da^{\ell+1}_i &=  \epsilon\log\left(\sum_j \exp{\Big(\log(\mu_1)_i-\left(\db^{\ell}_j-c_{i,j}\right)/\epsilon\Big)}\right)&i=1\cdots n_1\\
	\db^{\ell + 1}_j&= \epsilon\log\left(\sum_i \exp{\Big(\log(\mu_2)_j-\left(\da^{\ell+1}_i-c_{i,j}\right)/\epsilon\Big)}\right)&j=1\cdots n_2\label{algo_SK2b}
\end{align}
and we recover the SK algorithm \eqref{eq:update_SK} by taking $\uu_i=e^{\da_i/\epsilon}$, $\vv_j=e^{\db_j/\epsilon}$ and $\gamma^0_{i,j}=e^{-c_{i,j}/\epsilon}$.

Efficient parallel computations can be considered \cite{cuturi13} and one can almost reach real-time computation for large scale problem for certain class of cost matrices $c$  allowing the use of separable convolutions \cite{Solomon2015}. 
For low  values of the parameter $\epsilon$, numerical issues can arise and the log stabilization version of the algorithm presented in relations \eqref{algo_SK2} and \eqref{algo_SK2b} is necessary \cite{2016arXiv160705816C}.
Above all, the linear rate of convergence degrades as $\epsilon\to 0$ (see for instance Chapter 4 in  \cite{CuturiPeyre19}).
In the following sections, we introduce different numerical schemes that accelerate the convergence in the regime $\epsilon\to 0$.

\section{Regularized Nonlinear Acceleration of the  Sinkhorn-Knopp algorithm}\label{sec:RNA}
In order to accelerate the SK algorithm for low values of the regularization parameter $\epsilon$, we propose to rely on regularized nonlinear acceleration (RNA) techniques. 
In subsection \ref{sec:rna_def}, we first introduce RNA methods. The application to SK is then detailed in subsection \ref{sec:RNA4SK}.

\subsection{Regularized Nonlinear Acceleration}\label{sec:rna_def}

To introduce RNA, we first rewrite the SK algorithm \eqref{algo_SK}-\eqref{algo_SKb} as 
\begin{equation}\label{algo_SK_onestep}
\db^{\ell + 1}=\SK(\db^\ell):=\argmax_\db E\left(\argmax_\da E(\da,\db^\ell),\db\right).
\end{equation}
The goal of this algorithm is actually to build a sequence $(\beta^{\ell})_{l\geqslant 1}$ converging to a fixed point of SK, i.e. to a point $\beta^*$ satisfying 
\begin{equation}
\beta^*=SK(\beta^*).
\end{equation}
Actually many optimization problem can be recasted as fixed point problems. 
The Anderson acceleration or Anderson mixing, is a classical method to build a sequence $(x^n)_n$ that converges numerically fast to a fixed point of any operator $T$ from $\mathbb{R}^N$ to $\mathbb{R}^N$. This method defines at each step a linear but not necessary convex combination of some previous values of $(x^k)_k$ and $(Tx^k)_k$ to provide a value of $x^n$ such that $\norm{x^n-Tx^n}$ is as low as possible.

Numerically, fast local convergence rates can be observed when the operator $T$ is smooth. This method can nevertheless be unstable even in the  favourable setting where $T$ is affine. Such case arises for instance when minimizing  a quadratic function  $F$ with the descent operator $T=I-h\nabla F$, with time step $h>0$.   
Unfortunately there are no convergence guarantees, in a general setting or in the case $T=SK$, that the RNA sequence $(x^n)_n$ converges for any starting point $x^0$.\\

RNA is an algorithm that can be seen as a generalization of  the Anderson acceleration. It can also be applied to any fixed point problem. 
The RNA method \cite{scieur2018online} applied to algorithm \eqref{algo_SK_onestep} using at each step the $N$ previous iterates is:
\begin{align}\label{algo_RNA_SK}
(\ww^{\ell+1}_{l})_{l=0}^{N-1}&=\argmin_{{\bf \ww;}\, \sum_{l=0}^{N-1}\ww_l=1}\mathcal{P}_\lambda\left({\bf \ww},(\db^{\ell-l}-y^{\ell-l})_{l=0}^{N-1}\right)\\
y^{\ell+1}&=\sum_{l=0}^{N-1}\ww_{l}^{\ell+1}(y^{\ell-l}+\omega(\db^{\ell-l}-y^{\ell-l}))\label{algo_RNA_SKb}\\
\db^{\ell+1}&=\SK(y^{\ell+1})\label{algo_RNA_SKc},
\end{align}
where ${\bf \ww}=(\ww_{l})_{l=0}^{N-1}$ are extrapolation weights and $\omega$ is a relaxation parameter. RNA uses the memory of the past trajectory when $N>1$. We can remark that for $N=1$, $\ww_0^\ell=1$ for all $\ell\geqslant 0$ and RNA is reduced to a simple relaxation parameterized by $\omega$ :
\begin{equation*}
y^{\ell+1}=y^\ell+\omega (\beta^{\ell}-y^{\ell}).
\end{equation*}
Let us now discuss the role of the  different RNA parameters $\omega$ and ${\bf \ww}$.\\

\noindent
{\bf Relaxation}. 
Taking origins  from  Richardson's method \cite{richardson1911ix}, relaxation  leads  to numerical convergence improvements in gradient descent schemes \cite{iutzeler2019generic}. 
Anderson suggests to underrelax with  $\omega\in(0;1]$, while the authors of \cite{scieur2018online} propose to take $\omega=1$.\\

\noindent
{\bf Extrapolation}. 
Let us define the residual $r(y)=\textrm{SK}(y)-y$. As the objective is to estimate the fixed point of SK, the  extrapolation step  builds a vector $y$ such that $||r(y)||$ is minimal.
A relevant guess of such $y$ is obtained by looking at  a linear combination of previous iterates that reaches this minimum. 
More precisely, RNA methods estimate the weight vector $(\ww^{\ell+1}_{l})_{l=0}^{N-1}$ as  the unique solution of the regularized  problem:
\begin{align}\label{pb_weights}
(\ww^{\ell+1}_{l})_{l=0}^{N-1}&=\argmin_{{\bf \ww;}\, \sum_{l=0}^{N-1}\ww_l=1}\mathcal{P}_\lambda({\bf \ww},R):=||R{\bf \ww}||^2+\lambda ||{\bf \ww}||^2\\
&=\frac{(^tRR+\lambda \mathrm{Id}_N)^{-1}\mathbf{1}_N}{\langle(^tRR+\lambda \mathrm{Id}_N)^{-1}\mathbf{1}_N,\mathbf{1}_N\rangle}.
\end{align}
where the columns of $R:=[r(y^\ell), \cdots ,r(y^{\ell+1-N})]$ are the  $N$ previous residuals. The regularization parameter $\lambda>0$ generalizes the original Anderson acceleration \cite{anderson1965iterative} introduced for  $\lambda=0$. Taking $\lambda>0$ indeed leads to a more stable numerical estimation of the extrapolation parameters.

\subsection{Application to SK}\label{sec:RNA4SK}

We now detail the whole Sinkhorn-Knopp algorithm using Regularized Nonlinear Acceleration, that is presented in Algorithm \ref{algo:SK_RNA}. 

In all our experiments corresponding to $N>1$, we consider a regularization $\lambda=1e-10$ for the weigth estimation \eqref{pb_weights} within the RNA scheme. 
For the SK algorithm, we  consider the log-stabilization implementation  proposed in \cite{2016arXiv160705816C,schmitzer2016stabilized,2017arXiv170801955S} to avoid numerical errors for low values of $\epsilon$. This algorithm acts on the dual variables $\da,\db$. We refer to the aforementioned papers for more  details.

As the Sinkhorn-Knopp  algorithm successively projects the matrix $\gamma^\ell$ onto the set of linear constraints $\Ccal_k$, $k=1,2$, we take as convergence criteria the error realized on the first marginal of the transport matrix $\gamma=\diag(\exp(\da/\epsilon)) \exp(-c/\epsilon)\diag(\exp(\db/\epsilon))$, that is $\sum_i|\sum_j\gamma_{i,j}^\ell -(\mu_1)_i|<\eta=1e-9$. 
Notice that the variable $\da$ is  just introduced in the algorithm for computing the convergence criteria. 
\begin{algorithm}
	\caption{RNA SK algorithm in the log domain}
	\label{algo:SK_RNA}
	\begin{algorithmic}
		\REQUIRE $\mu_1\in \IR^{n_1}$, $\mu_2\in \IR^{n_2}$, $c\in \IR^{n_1\times n_2}_+$
		
		\STATE Set $\ell=0$, $\db^0,y^0=\mathbf{0}_{n_2}$, $\gamma^0=\exp{-c/\epsilon}$, $\omega\in\IR$, $N>0$ and $\eta>0$
		\STATE Set $x_i=-\max_j({c_{i,j}-y^{0}_j})$ and $\da_i=-\max_j({c_{i,j}-\db^{0}_j})$,\hfill$i=1\cdots n_1$
		\WHILE {$||\exp(\da/\epsilon)\otimes( \gamma^0 \exp(\db^\ell/\epsilon)) -  \mu_1||>\eta$}
		\STATE $\tilde N=\min(N,\ell+1)$
		\STATE $R= [\db^{\ell}-y^{\ell},\cdots ,\db^{\ell+1-{\tilde N}}-y^{\ell+1-{\tilde N}}]$
	 \STATE ${\bf w}=(^tRR+\lambda \mathrm{Id}_{\tilde N})^{-1}\mathbf{1}_{\tilde N}/\langle(^tRR+\lambda \mathrm{Id}_{\tilde N})^{-1}\mathbf{1}_{\tilde N},\mathbf{1}_{\tilde N}\rangle$
\STATE $y^{\ell+1}=\sum_{l=0}^{{\tilde N}-1}w_{l}((1-\omega)y^{\ell-l}+\omega\db^{\ell-l})$
		\STATE $\tilde x_i=-\max_j({c_{i,j}-y^{\ell+1}_j})$,\hfill$i=1\cdots n_1$
		\STATE $x_i=\tilde x_i-\epsilon\log(\sum_j \exp{((-c_{i,j}+\tilde x_i+y^{\ell+1}_j)/\epsilon}-\log(\mu_1)_i))$,\hfill $i=1\cdots n_1$
		\STATE $\db^{\ell+1}_j=y^{\ell+1}-\epsilon\log(\sum_i \exp{((-c_{i,j}+ x_i+y^{\ell+1}_j)/\epsilon}-\log(\mu_2)_j))$,\hfill $j=1\cdots n_2$	
		\STATE $\da_i=-\max_j({c_{i,j}-\db^{\ell+1}_j})$,\hfill$i=1\cdots n_1$
		\STATE $\ell\leftarrow\ell+1$
		\ENDWHILE

		\RETURN $\gamma_{i,j}=\diag(\exp(\da/\epsilon))\gamma^0\diag(\exp(\db/\epsilon))$
	\end{algorithmic}
\end{algorithm}

We now present  numerical results obtained with random cost matrices  of size $100\times 100$ with  entries uniform in $[0,1]$ and uniform marginals $\mu_1$ and $\mu_2$. 
All convergence plots are mean results over $20$ realizations.\\

We first consider the relaxation parameter $\omega=1$, in order to recover the original SK algorithm for  $N=1$. 
In Figure \ref{fig:RNA1}, we show convergence results obtained with RNA orders $N\in\{1,2,4,8\}$ on $4$ regularized transport problems corresponding to entropic  parameters $\epsilon\in\{0.003, 0.01,0.03,0.1\}$. 
Figure \ref{fig:RNA1} first illustrates that the convergence is improved with higher RNA orders $N$.

The acceleration is also larger for low values of the regularization parameter $\epsilon$. This is an important behaviour as a lot of iterations are required to have an accurate estimation of these  challenging  regularized transport problems. In the settings $\epsilon\in\{0.003,0.01\}$ a speed up of more than $\times 100$ in term of iteration number  is observed between RNA orders $N=8$ and $N=1$ (SK) to reach the same convergence threshold. We did not observe a significant improvement by considering higher RNA orders such as $N=16$. \\

\begin{figure}[ht!]
    \centering
    \begin{tabular}{cc}
    \includegraphics[width=0.465\textwidth]{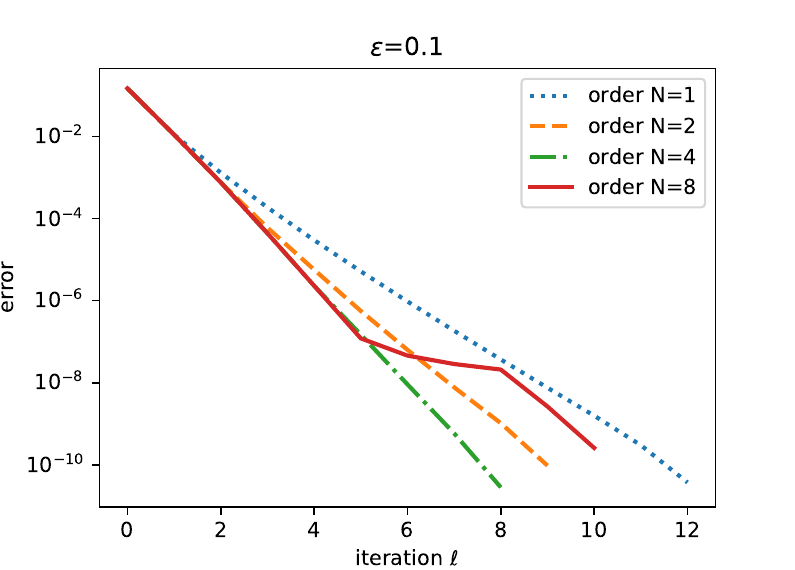}&
    \includegraphics[width=0.465\textwidth]{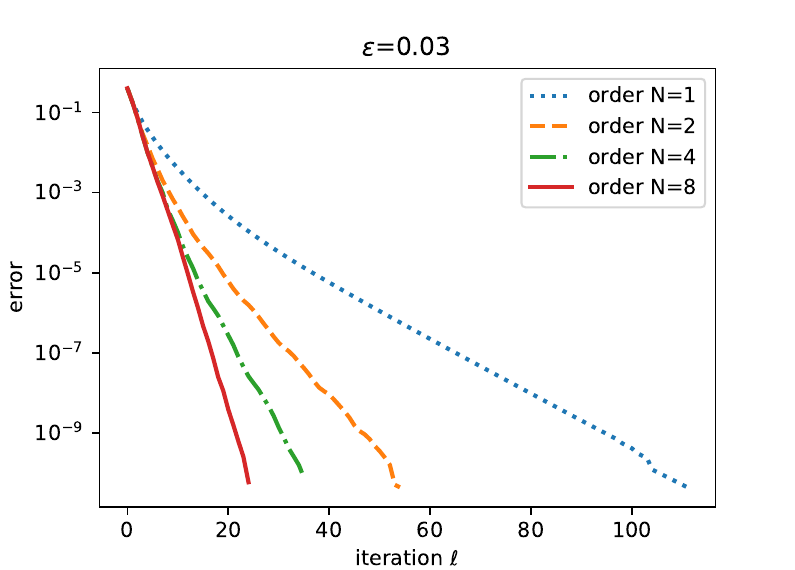}\\
    \includegraphics[width=0.465\textwidth]{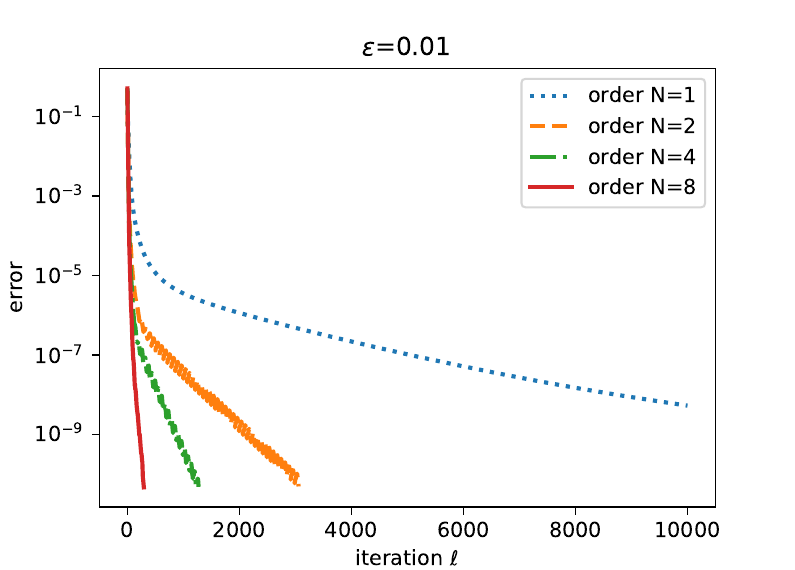}&
    \includegraphics[width=0.465\textwidth]{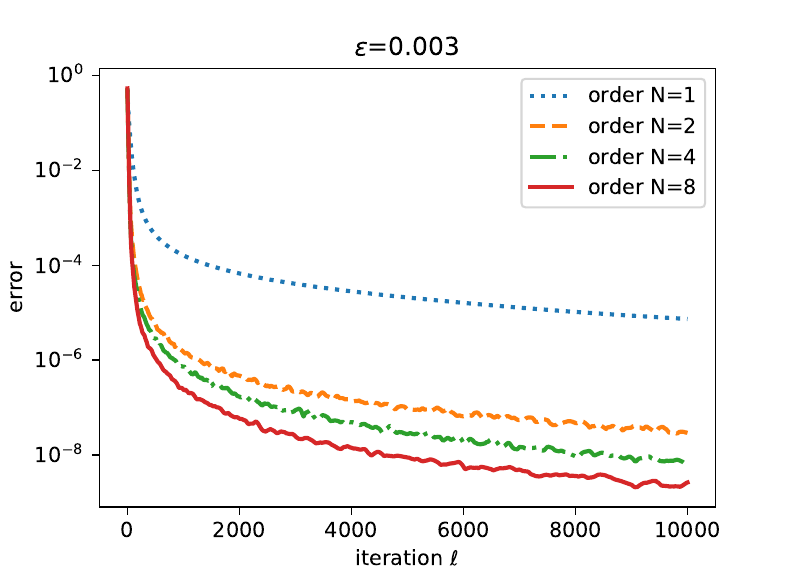}
    \end{tabular}
    \caption{Convergence of RNA schemes for a relaxation parameter $\omega=1$ and different orders $N\in\{1,2,4,8\}$. All approaches lead to  similar convergence than  the original SK algorithm (order $N=1$, with blue dots) for high values of the entropic  parameter such as $\epsilon=0.1$. When facing more challenging regularized optimal transport problems, high order RNA schemes ($N>1$) realize important accelerations. This behaviour is highlighted on the bottom row for $\epsilon=0.01$ and $\epsilon=0.003$. In these settings, with respect to SK,  RNA of  order $N=8$ (plain red curves) improves  with a factor $\times 100$ the number of iterations required to reach the same accuracy.}
    \label{fig:RNA1}
\end{figure}

Next, we focus on the influence of the relaxation parameter $\omega\in\{0.5,1,1.5, 1.9\}$ 
onto the behaviour of RNA schemes of orders  $N\in\{1,2,8\}$. We restrict our analysis to the more challenging settings $\epsilon=0.01$ and $\epsilon=0.003$.
As illustrated in Figure \ref{fig:RNA2}, increasing $\omega$ systematically leads to improvements in the case $N=1$. For other RNA orders satisfying $N>1$, we did not observe clear tendencies. Taking $\omega\approx 1.5$ generally allows to accelerate the convergence. 

We recall that the convergence of such approaches is not ensured.
This last experiment nevertheless suggests than in the case $N=1$, there is room to accelerate the original SK algorithm  ($\omega=1$), while keeping its global convergence guarantees, by looking at overrelaxed schemes with parameters $\omega>1$.

\newcommand{\sidecapY}[1]{ \begin{sideways}\parbox{80pt}{\centering #1}\end{sideways} }
\begin{figure}[ht!]
  \centering\hspace*{-0.1cm}
    \begin{tabular}{cccc}
    \hspace*{-0.2cm}\sidecapY{$\epsilon=0.01$}&\hspace{-0.2cm}
        \begin{adjustbox}{trim=7 2 14 2,clip}\begin{tikzpicture}[spy using outlines={rectangle, white,magnification=3, connect spies}]
		\node {\pgfimage[interpolate=true,width=0.345\textwidth]{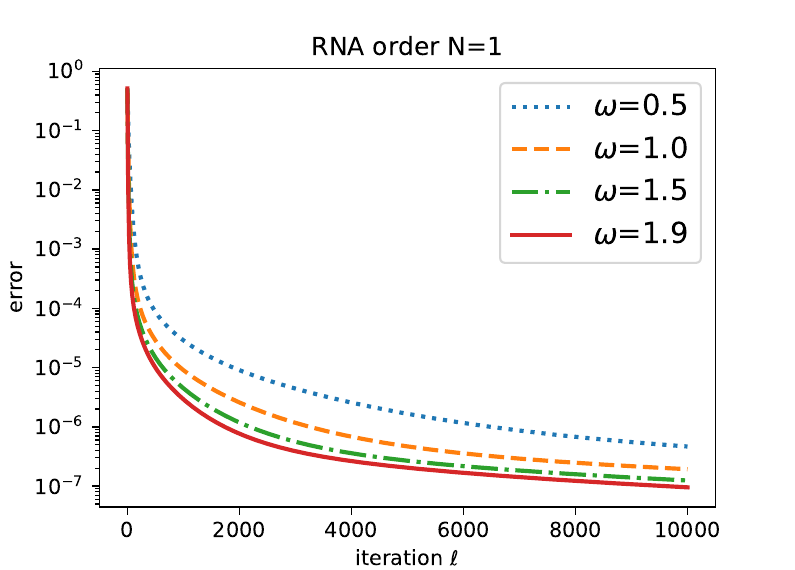}};
			\coordinate (spypoint) at (0.78,-.965);
		\coordinate (spyviewer) at (-0.3,0.45);
		\spy[black,width=1.35cm,height=1.35cm]on (spypoint) in node [fill=white] at (spyviewer);
		\end{tikzpicture}\end{adjustbox}
    &\hspace{-0.5cm}	\begin{adjustbox}{trim=7 2 14 2,clip}\begin{tikzpicture}[spy using outlines={rectangle, white,magnification=3, connect spies}]
		\node {\pgfimage[interpolate=true,width=0.345\textwidth]{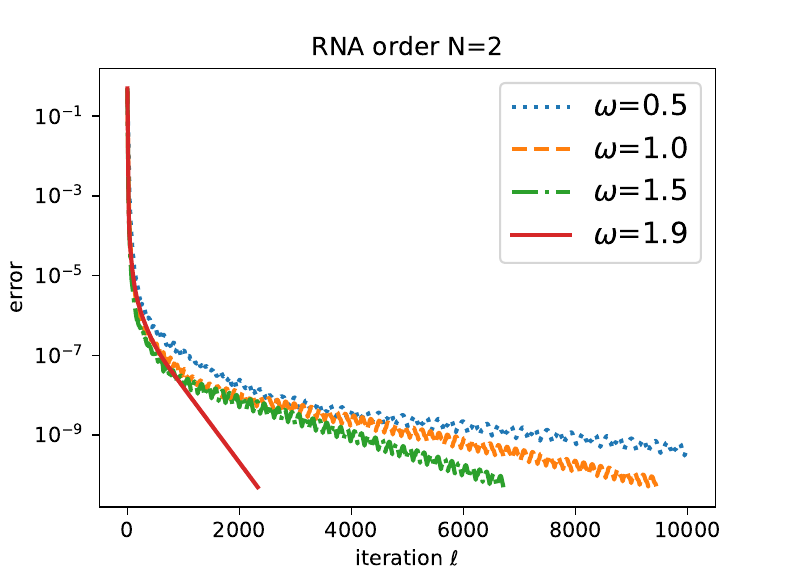}};
			\coordinate (spypoint) at (-1.2,-0.6);
		\coordinate (spyviewer) at (-0.3,0.45);
		\spy[black,width=1.35cm,height=1.35cm]on (spypoint) in node [fill=white] at (spyviewer);
		\end{tikzpicture}\end{adjustbox}&\hspace{-0.5cm}	\begin{adjustbox}{trim=7 2 14 2,clip}\begin{tikzpicture}[spy using outlines={rectangle, white,magnification=3, connect spies}]
		\node {\pgfimage[interpolate=true,width=0.345\textwidth]{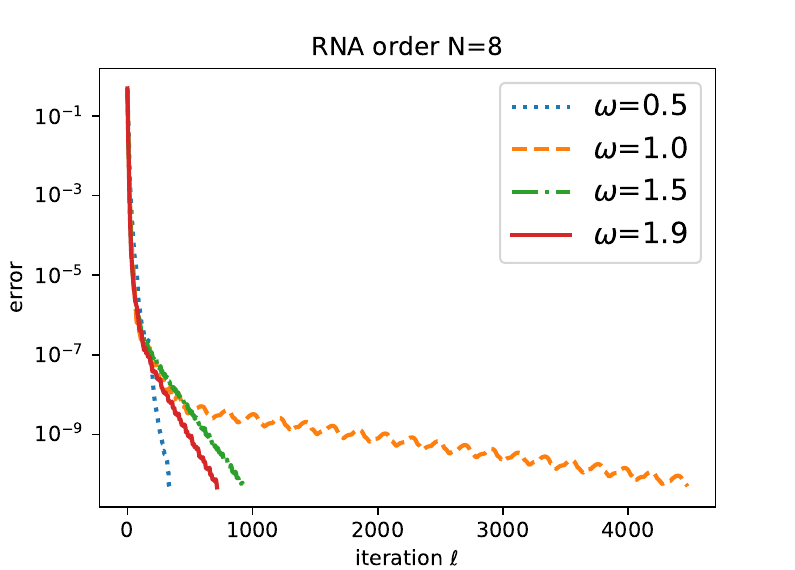}};
			\coordinate (spypoint) at (-1.25,-0.75);
		\coordinate (spyviewer) at (-0.3,0.45);
		\spy[black,width=1.35cm,height=1.35cm]on (spypoint) in node [fill=white] at (spyviewer);
		\end{tikzpicture}\end{adjustbox}\\
    \hspace*{-0.2cm}\sidecapY{$\epsilon=0.003$}&\hspace{-0.2cm}
    \begin{adjustbox}{trim=7 2 14 2,clip}\begin{tikzpicture}[spy using outlines={rectangle, white,magnification=3, connect spies}]
		\node {\pgfimage[interpolate=true,width=0.345\textwidth]{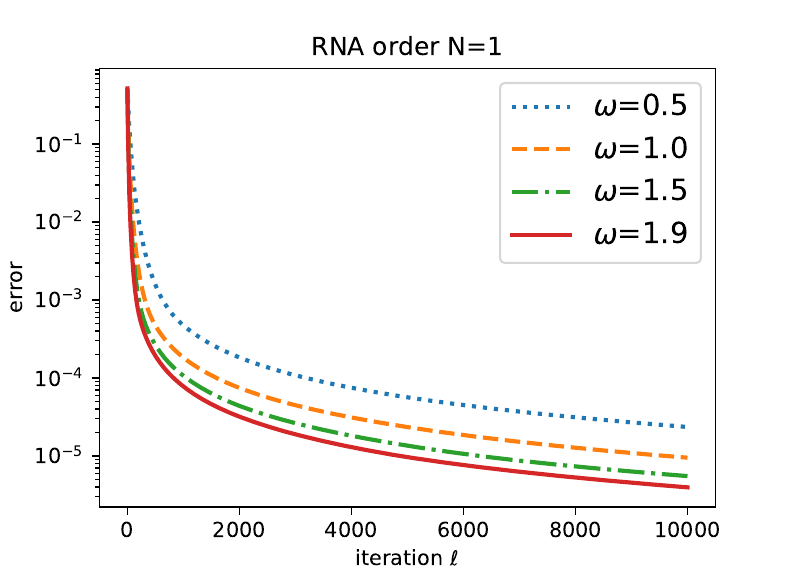}};
			\coordinate (spypoint) at (0.68,-.905);
		\coordinate (spyviewer) at (-0.3,0.45);
		\spy[black,width=1.35cm,height=1.35cm]on (spypoint) in node [fill=white] at (spyviewer);
		\end{tikzpicture}\end{adjustbox}
    &\hspace{-0.5cm}
        \begin{adjustbox}{trim=7 2 14 2,clip}\begin{tikzpicture}[spy using outlines={rectangle, white,magnification=3, connect spies}]
		\node {\pgfimage[interpolate=true,width=0.345\textwidth]{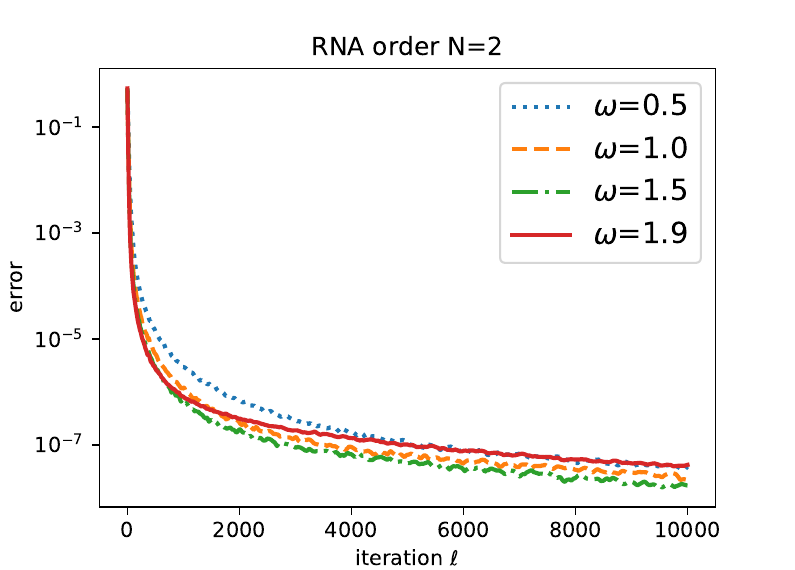}};
			\coordinate (spypoint) at (0.98,-1.09);
		\coordinate (spyviewer) at (-0.3,0.45);
		\spy[black,width=1.35cm,height=1.35cm]on (spypoint) in node [fill=white] at (spyviewer);
		\end{tikzpicture}\end{adjustbox}&\hspace{-0.5cm}
        \begin{adjustbox}{trim=7 2 14 2,clip}\begin{tikzpicture}[spy using outlines={rectangle, white,magnification=3, connect spies}]
		\node {\pgfimage[interpolate=true,width=0.345\textwidth]{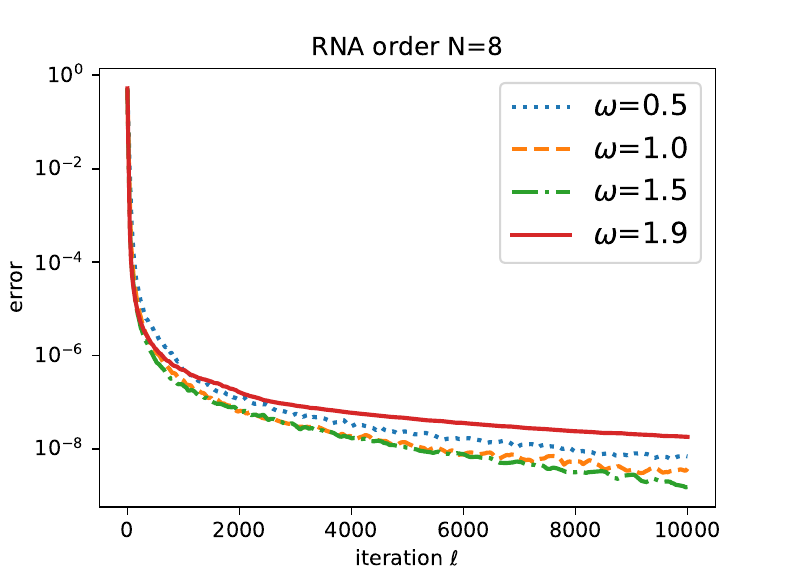}};
			\coordinate (spypoint) at (0.78,-.945);
		\coordinate (spyviewer) at (-0.3,0.45);
		\spy[black,width=1.35cm,height=1.35cm]on (spypoint) in node [fill=white] at (spyviewer);
		\end{tikzpicture}\end{adjustbox}\\
    &$N=1$&$N=2$&$N=8$
    
    \end{tabular}
    \caption{Comparison of RNA schemes on an optimal transport problem regularized with the entropic parameters $\epsilon=0.01$ (top line) and $\epsilon=0.003$ (bottom line). The convergence of  RNA schemes $N\in\{1,2,8\}$ is illustrated for different relaxation parameters $\omega\in\{0.5,1,1.5,1.9\}$. Higher values of $\omega$ lead to larger improvements in the case $N=1$ (first row). When $N>1$  (as in the middle row for $N=2$ the right row for $N=8$), it is not possible to conclude and to suggest a choice for the parameter $\omega$ with the obtained numerical results.}
    \label{fig:RNA2}
\end{figure}

\newpage
\subsection{Discussion}
The nonlinear regularized acceleration algorithm involves relevant numerical accelerations without convergence guarantees. 
To build an algorithm that ensures the convergence of iterates but also improves the numerical behavior of the SK algorithm, we now propose to follow a different approach using Lyapunov sequences, which is a classical tool to study optimization algorithms. The new scheme proposed here uses 
the specific form of the SK algorithm with a set up of the two variables $\da$ and $\db$. It performs two Successive OverRelaxions (SOR) at each step, one for the set up of $\da$ and one for $\db$. The algorithm does not use any mixing scheme but the simple structure allows to define a sequence, called a Lyapunov sequence, which decreases at each step. This Lyapunov approach allows to ensure the convergence of the algorithm for a suitable choice of the overrelaxation parameter.

The algorithm can be summarized as follow : 
\begin{align}\label{SORdual0}
	\da^{\ell+1} &= (1-\omega)\da^{\ell} + \omega \argmax_\da E(\da,\db^{\ell})\\
	\db^{\ell + 1}&=(1-\omega) \db^{\ell} + \omega \argmax_\db E(\da^{\ell+1},\db).
\end{align}

Our convergence analysis will rely on an online adaptation of an overrelaxation parameter $\omega\in[1,2)$. As illustrated by Figure \ref{fig:compRNA_SOR}, in the case for $\epsilon=0.01$, the proposed SOR method  is not as performant as high RNA orders $N>1$ with $\omega=1.5$. It nevertheless gives an important improvement with respect to the original SK method, while being provably convergent.
\begin{figure}[ht!]
    \centering
    \includegraphics[width=0.6\textwidth]{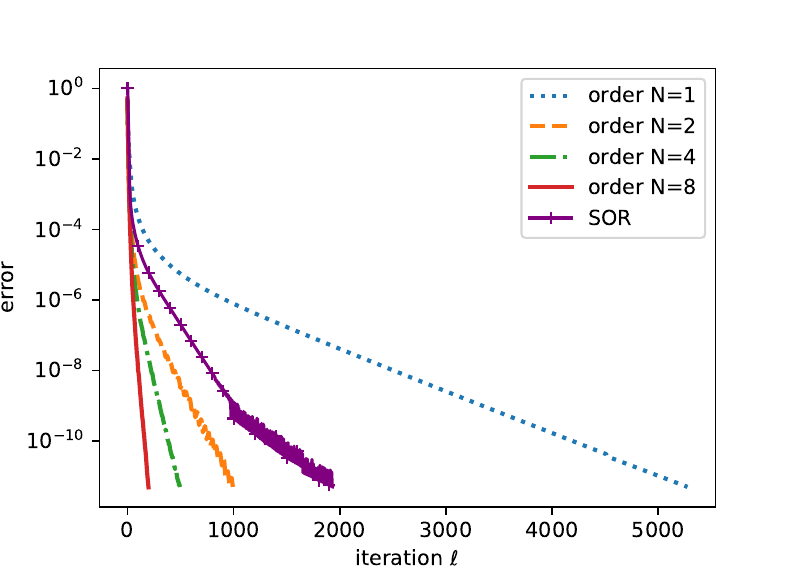}
    \caption{Comparison between RNA schemes with $\omega=1.5$ and SOR for a transport regularized with $\epsilon=0.01$. The SOR  performance is in between the ones of RNA of orders $N=1$ and $N=2$.}
    \label{fig:compRNA_SOR}
\end{figure}

\section{Overrelaxed Sinkhorn--Knopp algorithm}
In this section, we propose a globally convergent overrelaxed SK algorithm. Different from the RNA point of view of the previous section, our algorithm relies on successive overrelaxed (SOR) projections.

As illustrated in Figure \ref{alternate_projections} (a-b), the original SK algorithm \eqref{eq:update_SK} performs alternate Bregman projections \eqref{scaling} onto the affine sets $\Ccal_1$ and $\Ccal_2$. In practice, the convergence may degrade  when $\epsilon\to 0$. The idea developed in this section  is to perform overrelaxed projections in order to accelerate the process, as displayed in Figure \ref{alternate_projections} (c).

\begin{figure}[ht!]
	\centering
	\begin{minipage}[b]{.33\textwidth}
		\centering
		\begin{tikzpicture}
\fill (-0.500000,4.000000) circle (2pt) node[above] (gamma0) {$\gamma^0$};
\fill (0,0) circle (2pt) node[right] {$\gamma^*$};
\draw[dashed] (-0.500000,4.000000) -- (0.000000,4.000000)
(-1.723077,3.015385) -- (0.000000,4.000000)
(-1.723077,3.015385) -- (0.000000,3.015385)
(-1.298935,2.273136) -- (0.000000,3.015385)
(-1.298935,2.273136) -- (0.000000,2.273136)
(-0.979197,1.713595) -- (0.000000,2.273136)
(-0.979197,1.713595) -- (0.000000,1.713595)
(-0.738164,1.291787) -- (0.000000,1.713595)
(-0.738164,1.291787) -- (0.000000,1.291787)
(-0.556462,0.973809) -- (0.000000,1.291787)
(-0.556462,0.973809) -- (0.000000,0.973809)
(-0.419487,0.734102) -- (0.000000,0.973809)
(-0.419487,0.734102) -- (0.000000,0.734102)
(-0.316228,0.553400) -- (0.000000,0.734102)
(-0.316228,0.553400) -- (0.000000,0.553400)
(-0.238388,0.417178) -- (0.000000,0.553400)
(-0.238388,0.417178) -- (0.000000,0.417178)
(-0.179708,0.314488) -- (0.000000,0.417178)
(-0.179708,0.314488) -- (0.000000,0.314488)
(-0.135472,0.237076) -- (0.000000,0.314488)
(-0.135472,0.237076) -- (0.000000,0.237076)
(-0.102125,0.178719) -- (0.000000,0.237076)
(-0.102125,0.178719) -- (0.000000,0.178719)
(-0.076987,0.134726) -- (0.000000,0.178719)
(-0.076987,0.134726) -- (0.000000,0.134726)
(-0.058036,0.101563) -- (0.000000,0.134726)
(-0.058036,0.101563) -- (0.000000,0.101563)
(-0.043750,0.076563) -- (0.000000,0.101563)
(-0.043750,0.076563) -- (0.000000,0.076563)
(-0.032981,0.057717) -- (0.000000,0.076563)
(-0.032981,0.057717) -- (0.000000,0.057717)
(-0.024863,0.043509) -- (0.000000,0.057717)
(-0.024863,0.043509) -- (0.000000,0.043509)
(-0.018743,0.032799) -- (0.000000,0.043509)
(-0.018743,0.032799) -- (0.000000,0.032799)
(-0.014129,0.024726) -- (0.000000,0.032799)
(-0.014129,0.024726) -- (0.000000,0.024726)
(-0.010651,0.018639) -- (0.000000,0.024726)
(-0.010651,0.018639) -- (0.000000,0.018639)
(-0.008029,0.014051) -- (0.000000,0.018639)
(-0.008029,0.014051) -- (0.000000,0.014051)
(-0.006053,0.010592) -- (0.000000,0.014051)
(-0.006053,0.010592) -- (0.000000,0.010592)
(-0.004563,0.007985) -- (0.000000,0.010592)
(-0.004563,0.007985) -- (0.000000,0.007985)
(-0.003440,0.006020) -- (0.000000,0.007985)
(-0.003440,0.006020) -- (0.000000,0.006020)
(-0.002593,0.004538) -- (0.000000,0.006020)
(-0.002593,0.004538) -- (0.000000,0.004538)
(-0.001955,0.003421) -- (0.000000,0.004538)
(-0.001955,0.003421) -- (0.000000,0.003421)
(-0.001474,0.002579) -- (0.000000,0.003421)
(-0.001474,0.002579) -- (0.000000,0.002579)
(-0.001111,0.001944) -- (0.000000,0.002579)
(-0.001111,0.001944) -- (0.000000,0.001944)
(-0.000837,0.001465) -- (0.000000,0.001944)
(-0.000837,0.001465) -- (0.000000,0.001465)
(-0.000631,0.001105) -- (0.000000,0.001465)
(-0.000631,0.001105) -- (0.000000,0.001105)
(-0.000476,0.000833) -- (0.000000,0.001105)
(-0.000476,0.000833) -- (0.000000,0.000833)
(-0.000359,0.000628) -- (0.000000,0.000833)
(-0.000359,0.000628) -- (0.000000,0.000628)
(-0.000270,0.000473) -- (0.000000,0.000628)
(-0.000270,0.000473) -- (0.000000,0.000473)
(-0.000204,0.000357) -- (0.000000,0.000473)
(-0.000204,0.000357) -- (0.000000,0.000357)
(-0.000154,0.000269) -- (0.000000,0.000357)
(-0.000154,0.000269) -- (0.000000,0.000269)
(-0.000116,0.000203) -- (0.000000,0.000269)
(-0.000116,0.000203) -- (0.000000,0.000203)
(-0.000087,0.000153) -- (0.000000,0.000203)
(-0.000087,0.000153) -- (0.000000,0.000153)
(-0.000066,0.000115) -- (0.000000,0.000153)
(-0.000066,0.000115) -- (0.000000,0.000115)
(-0.000050,0.000087) -- (0.000000,0.000115)
(-0.000050,0.000087) -- (0.000000,0.000087)
(-0.000037,0.000065) -- (0.000000,0.000087)
(-0.000037,0.000065) -- (0.000000,0.000065)
(-0.000028,0.000049) -- (0.000000,0.000065)
(-0.000028,0.000049) -- (0.000000,0.000049)
(-0.000021,0.000037) -- (0.000000,0.000049)
(-0.000021,0.000037) -- (0.000000,0.000037)
(-0.000016,0.000028) -- (0.000000,0.000037)
(-0.000016,0.000028) -- (0.000000,0.000028)
(-0.000012,0.000021) -- (0.000000,0.000028)
(-0.000012,0.000021) -- (0.000000,0.000021)
(-0.000009,0.000016) -- (0.000000,0.000021)
(-0.000009,0.000016) -- (0.000000,0.000016)
(-0.000007,0.000012) -- (0.000000,0.000016)
(-0.000007,0.000012) -- (0.000000,0.000012)
(-0.000005,0.000009) -- (0.000000,0.000012)
(-0.000005,0.000009) -- (0.000000,0.000009)
(-0.000004,0.000007) -- (0.000000,0.000009)
(-0.000004,0.000007) -- (0.000000,0.000007)
(-0.000003,0.000005) -- (0.000000,0.000007)
(-0.000003,0.000005) -- (0.000000,0.000005)
(-0.000002,0.000004) -- (0.000000,0.000005)
(-0.000002,0.000004) -- (0.000000,0.000004)
(-0.000002,0.000003) -- (0.000000,0.000004)
(-0.000002,0.000003) -- (0.000000,0.000003)
(-0.000001,0.000002) -- (0.000000,0.000003)
(-0.000001,0.000002) -- (0.000000,0.000002)
(-0.000001,0.000002) -- (0.000000,0.000002)
(-0.000001,0.000002) -- (0.000000,0.000002)
(-0.000001,0.000001) -- (0.000000,0.000002)
(-0.000001,0.000001) -- (0.000000,0.000001)
(-0.000001,0.000001) -- (0.000000,0.000001)
(-0.000001,0.000001) -- (0.000000,0.000001)
(-0.000000,0.000001) -- (0.000000,0.000001)
(-0.000000,0.000001) -- (0.000000,0.000001)
(-0.000000,0.000001) -- (0.000000,0.000001)
(-0.000000,0.000001) -- (0.000000,0.000001)
(-0.000000,0.000000) -- (0.000000,0.000001)
(-0.000000,0.000000) -- (0.000000,0.000000)
(-0.000000,0.000000) -- (0.000000,0.000000)
(-0.000000,0.000000) -- (0.000000,0.000000)
(-0.000000,0.000000) -- (0.000000,0.000000)
(-0.000000,0.000000) -- (0.000000,0.000000)
(-0.000000,0.000000) -- (0.000000,0.000000)
(-0.000000,0.000000) -- (0.000000,0.000000)
(-0.000000,0.000000) -- (0.000000,0.000000)
(-0.000000,0.000000) -- (0.000000,0.000000)
(-0.000000,0.000000) -- (0.000000,0.000000)
(-0.000000,0.000000) -- (0.000000,0.000000)
(-0.000000,0.000000) -- (0.000000,0.000000)
(-0.000000,0.000000) -- (0.000000,0.000000)
(-0.000000,0.000000) -- (0.000000,0.000000)
(-0.000000,0.000000) -- (0.000000,0.000000)
(-0.000000,0.000000) -- (0.000000,0.000000)
(-0.000000,0.000000) -- (0.000000,0.000000)
(-0.000000,0.000000) -- (0.000000,0.000000)
(-0.000000,0.000000) -- (0.000000,0.000000)
(-0.000000,0.000000) -- (0.000000,0.000000)
(-0.000000,0.000000) -- (0.000000,0.000000)
(-0.000000,0.000000) -- (0.000000,0.000000)
(-0.000000,0.000000) -- (0.000000,0.000000)
(-0.000000,0.000000) -- (0.000000,0.000000)
(-0.000000,0.000000) -- (0.000000,0.000000)
(-0.000000,0.000000) -- (0.000000,0.000000)
(-0.000000,0.000000) -- (0.000000,0.000000)
(-0.000000,0.000000) -- (0.000000,0.000000)
(-0.000000,0.000000) -- (0.000000,0.000000)
(-0.000000,0.000000) -- (0.000000,0.000000)
(-0.000000,0.000000) -- (0.000000,0.000000)
(-0.000000,0.000000) -- (0.000000,0.000000)
(-0.000000,0.000000) -- (0.000000,0.000000)
(-0.000000,0.000000) -- (0.000000,0.000000)
(-0.000000,0.000000) -- (0.000000,0.000000)
(-0.000000,0.000000) -- (0.000000,0.000000)
(-0.000000,0.000000) -- (0.000000,0.000000)
(-0.000000,0.000000) -- (0.000000,0.000000)
(-0.000000,0.000000) -- (0.000000,0.000000)
(-0.000000,0.000000) -- (0.000000,0.000000)
(-0.000000,0.000000) -- (0.000000,0.000000)
(-0.000000,0.000000) -- (0.000000,0.000000)
(-0.000000,0.000000) -- (0.000000,0.000000)
(-0.000000,0.000000) -- (0.000000,0.000000)
(-0.000000,0.000000) -- (0.000000,0.000000)
(-0.000000,0.000000) -- (0.000000,0.000000)
(-0.000000,0.000000) -- (0.000000,0.000000)
(-0.000000,0.000000) -- (0.000000,0.000000)
(-0.000000,0.000000) -- (0.000000,0.000000)
(-0.000000,0.000000) -- (0.000000,0.000000)
(-0.000000,0.000000) -- (0.000000,0.000000)
(-0.000000,0.000000) -- (0.000000,0.000000)
(-0.000000,0.000000) -- (0.000000,0.000000)
(-0.000000,0.000000) -- (0.000000,0.000000)
(-0.000000,0.000000) -- (0.000000,0.000000)
(-0.000000,0.000000) -- (0.000000,0.000000)
(-0.000000,0.000000) -- (0.000000,0.000000)
(-0.000000,0.000000) -- (0.000000,0.000000)
(-0.000000,0.000000) -- (0.000000,0.000000)
(-0.000000,0.000000) -- (0.000000,0.000000)
(-0.000000,0.000000) -- (0.000000,0.000000)
(-0.000000,0.000000) -- (0.000000,0.000000)
(-0.000000,0.000000) -- (0.000000,0.000000)
(-0.000000,0.000000) -- (0.000000,0.000000)
(-0.000000,0.000000) -- (0.000000,0.000000)
(-0.000000,0.000000) -- (0.000000,0.000000)
(-0.000000,0.000000) -- (0.000000,0.000000)
(-0.000000,0.000000) -- (0.000000,0.000000)
(-0.000000,0.000000) -- (0.000000,0.000000)
(-0.000000,0.000000) -- (0.000000,0.000000)
(-0.000000,0.000000) -- (0.000000,0.000000)
(-0.000000,0.000000) -- (0.000000,0.000000)
(-0.000000,0.000000) -- (0.000000,0.000000)
(-0.000000,0.000000) -- (0.000000,0.000000)
(-0.000000,0.000000) -- (0.000000,0.000000)
(-0.000000,0.000000) -- (0.000000,0.000000)
(-0.000000,0.000000) -- (0.000000,0.000000)
(-0.000000,0.000000) -- (0.000000,0.000000)
(-0.000000,0.000000) -- (0.000000,0.000000)
(-0.000000,0.000000) -- (0.000000,0.000000)
(-0.000000,0.000000) -- (0.000000,0.000000)
(-0.000000,0.000000) -- (0.000000,0.000000)
(-0.000000,0.000000) -- (0.000000,0.000000)
(-0.000000,0.000000) -- (0.000000,0.000000)
(-0.000000,0.000000) -- (0.000000,0.000000)
(-0.000000,0.000000) -- (0.000000,0.000000)
;
\draw (-0.000000,-0.400000) -- (0.000000,4.400000) node[right] {$\mathcal{C}_1$};
\draw (0.228571,-0.400000) -- (-2.514286,4.400000) node[right] {$\mathcal{C}_2$};
\end{tikzpicture}
		\subcaption{}\label{fig:schema_a}
	\end{minipage}%
	\begin{minipage}[b]{.33\textwidth}
		\centering
		\begin{tikzpicture}
\fill (-0.500000,4.000000) circle (2pt) node[above] (gamma0) {$\gamma^0$};
\fill (0,0) circle (2pt) node[right] {$\gamma^*$};
\draw[dashed] (-0.500000,4.000000) -- (0.000000,4.000000)
(-0.995851,3.734440) -- (0.000000,4.000000)
(-0.995851,3.734440) -- (0.000000,3.734440)
(-0.929736,3.486510) -- (0.000000,3.734440)
(-0.929736,3.486510) -- (0.000000,3.486510)
(-0.868011,3.255041) -- (0.000000,3.486510)
(-0.868011,3.255041) -- (0.000000,3.255041)
(-0.810384,3.038938) -- (0.000000,3.255041)
(-0.810384,3.038938) -- (0.000000,3.038938)
(-0.756582,2.837183) -- (0.000000,3.038938)
(-0.756582,2.837183) -- (0.000000,2.837183)
(-0.706353,2.648822) -- (0.000000,2.837183)
(-0.706353,2.648822) -- (0.000000,2.648822)
(-0.659458,2.472967) -- (0.000000,2.648822)
(-0.659458,2.472967) -- (0.000000,2.472967)
(-0.615676,2.308787) -- (0.000000,2.472967)
(-0.615676,2.308787) -- (0.000000,2.308787)
(-0.574802,2.155506) -- (0.000000,2.308787)
(-0.574802,2.155506) -- (0.000000,2.155506)
(-0.536641,2.012402) -- (0.000000,2.155506)
(-0.536641,2.012402) -- (0.000000,2.012402)
(-0.501013,1.878799) -- (0.000000,2.012402)
(-0.501013,1.878799) -- (0.000000,1.878799)
(-0.467751,1.754065) -- (0.000000,1.878799)
(-0.467751,1.754065) -- (0.000000,1.754065)
(-0.436697,1.637613) -- (0.000000,1.754065)
(-0.436697,1.637613) -- (0.000000,1.637613)
(-0.407704,1.528891) -- (0.000000,1.637613)
(-0.407704,1.528891) -- (0.000000,1.528891)
(-0.380637,1.427388) -- (0.000000,1.528891)
(-0.380637,1.427388) -- (0.000000,1.427388)
(-0.355366,1.332624) -- (0.000000,1.427388)
(-0.355366,1.332624) -- (0.000000,1.332624)
(-0.331774,1.244151) -- (0.000000,1.332624)
(-0.331774,1.244151) -- (0.000000,1.244151)
(-0.309747,1.161552) -- (0.000000,1.244151)
(-0.309747,1.161552) -- (0.000000,1.161552)
(-0.289183,1.084436) -- (0.000000,1.161552)
(-0.289183,1.084436) -- (0.000000,1.084436)
(-0.269984,1.012440) -- (0.000000,1.084436)
(-0.269984,1.012440) -- (0.000000,1.012440)
(-0.252060,0.945225) -- (0.000000,1.012440)
(-0.252060,0.945225) -- (0.000000,0.945225)
(-0.235326,0.882471) -- (0.000000,0.945225)
(-0.235326,0.882471) -- (0.000000,0.882471)
(-0.219702,0.823884) -- (0.000000,0.882471)
(-0.219702,0.823884) -- (0.000000,0.823884)
(-0.205116,0.769186) -- (0.000000,0.823884)
(-0.205116,0.769186) -- (0.000000,0.769186)
(-0.191499,0.718120) -- (0.000000,0.769186)
(-0.191499,0.718120) -- (0.000000,0.718120)
(-0.178785,0.670444) -- (0.000000,0.718120)
(-0.178785,0.670444) -- (0.000000,0.670444)
(-0.166915,0.625933) -- (0.000000,0.670444)
(-0.166915,0.625933) -- (0.000000,0.625933)
(-0.155834,0.584377) -- (0.000000,0.625933)
(-0.155834,0.584377) -- (0.000000,0.584377)
(-0.145488,0.545580) -- (0.000000,0.584377)
(-0.145488,0.545580) -- (0.000000,0.545580)
(-0.135829,0.509359) -- (0.000000,0.545580)
(-0.135829,0.509359) -- (0.000000,0.509359)
(-0.126811,0.475543) -- (0.000000,0.509359)
(-0.126811,0.475543) -- (0.000000,0.475543)
(-0.118392,0.443972) -- (0.000000,0.475543)
(-0.118392,0.443972) -- (0.000000,0.443972)
(-0.110532,0.414496) -- (0.000000,0.443972)
(-0.110532,0.414496) -- (0.000000,0.414496)
(-0.103194,0.386978) -- (0.000000,0.414496)
(-0.103194,0.386978) -- (0.000000,0.386978)
(-0.096343,0.361286) -- (0.000000,0.386978)
(-0.096343,0.361286) -- (0.000000,0.361286)
(-0.089947,0.337301) -- (0.000000,0.361286)
(-0.089947,0.337301) -- (0.000000,0.337301)
(-0.083975,0.314907) -- (0.000000,0.337301)
(-0.083975,0.314907) -- (0.000000,0.314907)
(-0.078400,0.294000) -- (0.000000,0.314907)
(-0.078400,0.294000) -- (0.000000,0.294000)
(-0.073195,0.274482) -- (0.000000,0.294000)
(-0.073195,0.274482) -- (0.000000,0.274482)
(-0.068336,0.256259) -- (0.000000,0.274482)
(-0.068336,0.256259) -- (0.000000,0.256259)
(-0.063799,0.239246) -- (0.000000,0.256259)
(-0.063799,0.239246) -- (0.000000,0.239246)
(-0.059563,0.223362) -- (0.000000,0.239246)
(-0.059563,0.223362) -- (0.000000,0.223362)
(-0.055609,0.208533) -- (0.000000,0.223362)
(-0.055609,0.208533) -- (0.000000,0.208533)
(-0.051917,0.194689) -- (0.000000,0.208533)
(-0.051917,0.194689) -- (0.000000,0.194689)
(-0.048470,0.181763) -- (0.000000,0.194689)
(-0.048470,0.181763) -- (0.000000,0.181763)
(-0.045252,0.169696) -- (0.000000,0.181763)
(-0.045252,0.169696) -- (0.000000,0.169696)
(-0.042248,0.158430) -- (0.000000,0.169696)
(-0.042248,0.158430) -- (0.000000,0.158430)
(-0.039443,0.147912) -- (0.000000,0.158430)
(-0.039443,0.147912) -- (0.000000,0.147912)
(-0.036825,0.138092) -- (0.000000,0.147912)
(-0.036825,0.138092) -- (0.000000,0.138092)
(-0.034380,0.128924) -- (0.000000,0.138092)
(-0.034380,0.128924) -- (0.000000,0.128924)
(-0.032097,0.120365) -- (0.000000,0.128924)
(-0.032097,0.120365) -- (0.000000,0.120365)
(-0.029966,0.112374) -- (0.000000,0.120365)
(-0.029966,0.112374) -- (0.000000,0.112374)
(-0.027977,0.104913) -- (0.000000,0.112374)
(-0.027977,0.104913) -- (0.000000,0.104913)
(-0.026119,0.097948) -- (0.000000,0.104913)
(-0.026119,0.097948) -- (0.000000,0.097948)
(-0.024385,0.091445) -- (0.000000,0.097948)
(-0.024385,0.091445) -- (0.000000,0.091445)
(-0.022766,0.085374) -- (0.000000,0.091445)
(-0.022766,0.085374) -- (0.000000,0.085374)
(-0.021255,0.079706) -- (0.000000,0.085374)
(-0.021255,0.079706) -- (0.000000,0.079706)
(-0.019844,0.074415) -- (0.000000,0.079706)
(-0.019844,0.074415) -- (0.000000,0.074415)
(-0.018526,0.069474) -- (0.000000,0.074415)
(-0.018526,0.069474) -- (0.000000,0.069474)
(-0.017296,0.064862) -- (0.000000,0.069474)
(-0.017296,0.064862) -- (0.000000,0.064862)
(-0.016148,0.060556) -- (0.000000,0.064862)
(-0.016148,0.060556) -- (0.000000,0.060556)
(-0.015076,0.056535) -- (0.000000,0.060556)
(-0.015076,0.056535) -- (0.000000,0.056535)
(-0.014075,0.052782) -- (0.000000,0.056535)
(-0.014075,0.052782) -- (0.000000,0.052782)
(-0.013141,0.049278) -- (0.000000,0.052782)
(-0.013141,0.049278) -- (0.000000,0.049278)
(-0.012268,0.046006) -- (0.000000,0.049278)
(-0.012268,0.046006) -- (0.000000,0.046006)
(-0.011454,0.042952) -- (0.000000,0.046006)
(-0.011454,0.042952) -- (0.000000,0.042952)
(-0.010693,0.040100) -- (0.000000,0.042952)
(-0.010693,0.040100) -- (0.000000,0.040100)
(-0.009983,0.037438) -- (0.000000,0.040100)
(-0.009983,0.037438) -- (0.000000,0.037438)
(-0.009321,0.034952) -- (0.000000,0.037438)
(-0.009321,0.034952) -- (0.000000,0.034952)
(-0.008702,0.032632) -- (0.000000,0.034952)
(-0.008702,0.032632) -- (0.000000,0.032632)
(-0.008124,0.030466) -- (0.000000,0.032632)
(-0.008124,0.030466) -- (0.000000,0.030466)
(-0.007585,0.028443) -- (0.000000,0.030466)
(-0.007585,0.028443) -- (0.000000,0.028443)
(-0.007081,0.026555) -- (0.000000,0.028443)
(-0.007081,0.026555) -- (0.000000,0.026555)
(-0.006611,0.024792) -- (0.000000,0.026555)
(-0.006611,0.024792) -- (0.000000,0.024792)
(-0.006172,0.023146) -- (0.000000,0.024792)
(-0.006172,0.023146) -- (0.000000,0.023146)
(-0.005762,0.021609) -- (0.000000,0.023146)
(-0.005762,0.021609) -- (0.000000,0.021609)
(-0.005380,0.020174) -- (0.000000,0.021609)
(-0.005380,0.020174) -- (0.000000,0.020174)
(-0.005023,0.018835) -- (0.000000,0.020174)
(-0.005023,0.018835) -- (0.000000,0.018835)
(-0.004689,0.017585) -- (0.000000,0.018835)
(-0.004689,0.017585) -- (0.000000,0.017585)
(-0.004378,0.016417) -- (0.000000,0.017585)
(-0.004378,0.016417) -- (0.000000,0.016417)
(-0.004087,0.015327) -- (0.000000,0.016417)
(-0.004087,0.015327) -- (0.000000,0.015327)
(-0.003816,0.014310) -- (0.000000,0.015327)
(-0.003816,0.014310) -- (0.000000,0.014310)
(-0.003563,0.013360) -- (0.000000,0.014310)
(-0.003563,0.013360) -- (0.000000,0.013360)
(-0.003326,0.012473) -- (0.000000,0.013360)
(-0.003326,0.012473) -- (0.000000,0.012473)
(-0.003105,0.011645) -- (0.000000,0.012473)
(-0.003105,0.011645) -- (0.000000,0.011645)
(-0.002899,0.010872) -- (0.000000,0.011645)
(-0.002899,0.010872) -- (0.000000,0.010872)
(-0.002707,0.010150) -- (0.000000,0.010872)
(-0.002707,0.010150) -- (0.000000,0.010150)
(-0.002527,0.009476) -- (0.000000,0.010150)
(-0.002527,0.009476) -- (0.000000,0.009476)
(-0.002359,0.008847) -- (0.000000,0.009476)
(-0.002359,0.008847) -- (0.000000,0.008847)
(-0.002203,0.008259) -- (0.000000,0.008847)
(-0.002203,0.008259) -- (0.000000,0.008259)
(-0.002056,0.007711) -- (0.000000,0.008259)
(-0.002056,0.007711) -- (0.000000,0.007711)
(-0.001920,0.007199) -- (0.000000,0.007711)
(-0.001920,0.007199) -- (0.000000,0.007199)
(-0.001792,0.006721) -- (0.000000,0.007199)
(-0.001792,0.006721) -- (0.000000,0.006721)
(-0.001673,0.006275) -- (0.000000,0.006721)
(-0.001673,0.006275) -- (0.000000,0.006275)
(-0.001562,0.005858) -- (0.000000,0.006275)
(-0.001562,0.005858) -- (0.000000,0.005858)
(-0.001459,0.005469) -- (0.000000,0.005858)
(-0.001459,0.005469) -- (0.000000,0.005469)
(-0.001362,0.005106) -- (0.000000,0.005469)
(-0.001362,0.005106) -- (0.000000,0.005106)
(-0.001271,0.004767) -- (0.000000,0.005106)
(-0.001271,0.004767) -- (0.000000,0.004767)
(-0.001187,0.004451) -- (0.000000,0.004767)
(-0.001187,0.004451) -- (0.000000,0.004451)
(-0.001108,0.004155) -- (0.000000,0.004451)
;
\draw (-0.000000,-0.400000) -- (0.000000,4.400000) node[right] {$\mathcal{C}_1$};
\draw (0.106667,-0.400000) -- (-1.173333,4.400000) node[right] {$\mathcal{C}_2$};
\end{tikzpicture}
		\subcaption{}\label{fig:schema_b}
	\end{minipage}%
	\begin{minipage}[b]{.33\textwidth}
		\centering
		\begin{tikzpicture}
\fill (-0.500000,4.000000) circle (2pt) node[above] (gamma0) {$\gamma^0$};
\fill (0,0) circle (2pt) node[right] {$\gamma^*$};
\draw[dashed] (-0.500000,4.000000) -- (0.350000,4.000000)
(-1.898444,3.400415) -- (0.350000,4.000000)
(-1.898444,3.400415) -- (1.328911,3.400415)
(-2.219432,2.454190) -- (1.328911,3.400415)
(-2.219432,2.454190) -- (1.553603,2.454190)
(-1.950880,1.519661) -- (1.553603,2.454190)
(-1.950880,1.519661) -- (1.365616,1.519661)
(-1.444980,0.770169) -- (1.365616,1.519661)
(-1.444980,0.770169) -- (1.011486,0.770169)
(-0.919844,0.255148) -- (1.011486,0.770169)
(-0.919844,0.255148) -- (0.643891,0.255148)
(-0.486040,-0.046167) -- (0.643891,0.255148)
(-0.486040,-0.046167) -- (0.340228,-0.046167)
(-0.180221,-0.184954) -- (0.340228,-0.046167)
(-0.180221,-0.184954) -- (0.126154,-0.184954)
(0.004209,-0.217472) -- (0.126154,-0.184954)
(0.004209,-0.217472) -- (-0.002946,-0.217472)
(0.093772,-0.191681) -- (-0.002946,-0.217472)
(0.093772,-0.191681) -- (-0.065641,-0.191681)
(0.119666,-0.142266) -- (-0.065641,-0.191681)
(0.119666,-0.142266) -- (-0.083766,-0.142266)
(0.109394,-0.090756) -- (-0.083766,-0.142266)
(0.109394,-0.090756) -- (-0.076576,-0.090756)
(0.083372,-0.048103) -- (-0.076576,-0.090756)
(0.083372,-0.048103) -- (-0.058360,-0.048103)
(0.054625,-0.017974) -- (-0.058360,-0.048103)
(0.054625,-0.017974) -- (-0.038237,-0.017974)
(0.030058,0.000238) -- (-0.038237,-0.017974)
(0.030058,0.000238) -- (-0.021040,0.000238)
(0.012253,0.009116) -- (-0.021040,0.000238)
(0.012253,0.009116) -- (-0.008577,0.009116)
(0.001178,0.011717) -- (-0.008577,0.009116)
(0.001178,0.011717) -- (-0.000824,0.011717)
(-0.004475,0.010744) -- (-0.000824,0.011717)
(-0.004475,0.010744) -- (0.003133,0.010744)
(-0.006387,0.008205) -- (0.003133,0.010744)
(-0.006387,0.008205) -- (0.004471,0.008205)
(-0.006098,0.005387) -- (0.004471,0.008205)
(-0.006098,0.005387) -- (0.004268,0.005387)
(-0.004786,0.002973) -- (0.004268,0.005387)
(-0.004786,0.002973) -- (0.003350,0.002973)
(-0.003225,0.001219) -- (0.003350,0.002973)
(-0.003225,0.001219) -- (0.002258,0.001219)
(-0.001842,0.000126) -- (0.002258,0.001219)
(-0.001842,0.000126) -- (0.001289,0.000126)
(-0.000810,-0.000434) -- (0.001289,0.000126)
(-0.000810,-0.000434) -- (0.000567,-0.000434)
(-0.000149,-0.000625) -- (0.000567,-0.000434)
(-0.000149,-0.000625) -- (0.000105,-0.000625)
(0.000203,-0.000599) -- (0.000105,-0.000625)
(0.000203,-0.000599) -- (-0.000142,-0.000599)
(0.000337,-0.000471) -- (-0.000142,-0.000599)
(0.000337,-0.000471) -- (-0.000236,-0.000471)
(0.000338,-0.000318) -- (-0.000236,-0.000471)
(0.000338,-0.000318) -- (-0.000236,-0.000318)
(0.000273,-0.000182) -- (-0.000236,-0.000318)
(0.000273,-0.000182) -- (-0.000191,-0.000182)
(0.000189,-0.000080) -- (-0.000191,-0.000182)
(0.000189,-0.000080) -- (-0.000133,-0.000080)
(0.000112,-0.000015) -- (-0.000133,-0.000080)
(0.000112,-0.000015) -- (-0.000078,-0.000015)
(0.000052,0.000020) -- (-0.000078,-0.000015)
(0.000052,0.000020) -- (-0.000037,0.000020)
(0.000013,0.000033) -- (-0.000037,0.000020)
(0.000013,0.000033) -- (-0.000009,0.000033)
(-0.000008,0.000033) -- (-0.000009,0.000033)
(-0.000008,0.000033) -- (0.000006,0.000033)
(-0.000018,0.000027) -- (0.000006,0.000033)
(-0.000018,0.000027) -- (0.000012,0.000027)
(-0.000019,0.000019) -- (0.000012,0.000027)
(-0.000019,0.000019) -- (0.000013,0.000019)
(-0.000016,0.000011) -- (0.000013,0.000019)
(-0.000016,0.000011) -- (0.000011,0.000011)
(-0.000011,0.000005) -- (0.000011,0.000011)
(-0.000011,0.000005) -- (0.000008,0.000005)
(-0.000007,0.000001) -- (0.000008,0.000005)
(-0.000007,0.000001) -- (0.000005,0.000001)
(-0.000003,-0.000001) -- (0.000005,0.000001)
(-0.000003,-0.000001) -- (0.000002,-0.000001)
(-0.000001,-0.000002) -- (0.000002,-0.000001)
(-0.000001,-0.000002) -- (0.000001,-0.000002)
(0.000000,-0.000002) -- (0.000001,-0.000002)
(0.000000,-0.000002) -- (-0.000000,-0.000002)
(0.000001,-0.000002) -- (-0.000000,-0.000002)
(0.000001,-0.000002) -- (-0.000001,-0.000002)
(0.000001,-0.000001) -- (-0.000001,-0.000002)
(0.000001,-0.000001) -- (-0.000001,-0.000001)
(0.000001,-0.000001) -- (-0.000001,-0.000001)
(0.000001,-0.000001) -- (-0.000001,-0.000001)
(0.000001,-0.000000) -- (-0.000001,-0.000001)
(0.000001,-0.000000) -- (-0.000000,-0.000000)
(0.000000,-0.000000) -- (-0.000000,-0.000000)
(0.000000,-0.000000) -- (-0.000000,-0.000000)
(0.000000,0.000000) -- (-0.000000,-0.000000)
(0.000000,0.000000) -- (-0.000000,0.000000)
(0.000000,0.000000) -- (-0.000000,0.000000)
(0.000000,0.000000) -- (-0.000000,0.000000)
(-0.000000,0.000000) -- (-0.000000,0.000000)
(-0.000000,0.000000) -- (0.000000,0.000000)
(-0.000000,0.000000) -- (0.000000,0.000000)
(-0.000000,0.000000) -- (0.000000,0.000000)
(-0.000000,0.000000) -- (0.000000,0.000000)
(-0.000000,0.000000) -- (0.000000,0.000000)
(-0.000000,0.000000) -- (0.000000,0.000000)
(-0.000000,0.000000) -- (0.000000,0.000000)
(-0.000000,0.000000) -- (0.000000,0.000000)
(-0.000000,0.000000) -- (0.000000,0.000000)
(-0.000000,0.000000) -- (0.000000,0.000000)
(-0.000000,0.000000) -- (0.000000,0.000000)
(-0.000000,-0.000000) -- (0.000000,0.000000)
(-0.000000,-0.000000) -- (0.000000,-0.000000)
(-0.000000,-0.000000) -- (0.000000,-0.000000)
(-0.000000,-0.000000) -- (0.000000,-0.000000)
(-0.000000,-0.000000) -- (0.000000,-0.000000)
(-0.000000,-0.000000) -- (0.000000,-0.000000)
(0.000000,-0.000000) -- (0.000000,-0.000000)
(0.000000,-0.000000) -- (-0.000000,-0.000000)
(0.000000,-0.000000) -- (-0.000000,-0.000000)
(0.000000,-0.000000) -- (-0.000000,-0.000000)
(0.000000,-0.000000) -- (-0.000000,-0.000000)
(0.000000,-0.000000) -- (-0.000000,-0.000000)
(0.000000,-0.000000) -- (-0.000000,-0.000000)
(0.000000,-0.000000) -- (-0.000000,-0.000000)
(0.000000,-0.000000) -- (-0.000000,-0.000000)
(0.000000,-0.000000) -- (-0.000000,-0.000000)
(0.000000,-0.000000) -- (-0.000000,-0.000000)
(0.000000,-0.000000) -- (-0.000000,-0.000000)
(0.000000,0.000000) -- (-0.000000,-0.000000)
(0.000000,0.000000) -- (-0.000000,0.000000)
(0.000000,0.000000) -- (-0.000000,0.000000)
(0.000000,0.000000) -- (-0.000000,0.000000)
(-0.000000,0.000000) -- (-0.000000,0.000000)
(-0.000000,0.000000) -- (0.000000,0.000000)
(-0.000000,0.000000) -- (0.000000,0.000000)
(-0.000000,0.000000) -- (0.000000,0.000000)
(-0.000000,0.000000) -- (0.000000,0.000000)
(-0.000000,0.000000) -- (0.000000,0.000000)
(-0.000000,0.000000) -- (0.000000,0.000000)
(-0.000000,0.000000) -- (0.000000,0.000000)
(-0.000000,0.000000) -- (0.000000,0.000000)
(-0.000000,0.000000) -- (0.000000,0.000000)
(-0.000000,0.000000) -- (0.000000,0.000000)
(-0.000000,0.000000) -- (0.000000,0.000000)
(-0.000000,-0.000000) -- (0.000000,0.000000)
(-0.000000,-0.000000) -- (0.000000,-0.000000)
(-0.000000,-0.000000) -- (0.000000,-0.000000)
(-0.000000,-0.000000) -- (0.000000,-0.000000)
(0.000000,-0.000000) -- (0.000000,-0.000000)
(0.000000,-0.000000) -- (-0.000000,-0.000000)
(0.000000,-0.000000) -- (-0.000000,-0.000000)
(0.000000,-0.000000) -- (-0.000000,-0.000000)
(0.000000,-0.000000) -- (-0.000000,-0.000000)
(0.000000,-0.000000) -- (-0.000000,-0.000000)
(0.000000,-0.000000) -- (-0.000000,-0.000000)
(0.000000,-0.000000) -- (-0.000000,-0.000000)
(0.000000,-0.000000) -- (-0.000000,-0.000000)
(0.000000,-0.000000) -- (-0.000000,-0.000000)
(0.000000,-0.000000) -- (-0.000000,-0.000000)
(0.000000,-0.000000) -- (-0.000000,-0.000000)
(0.000000,0.000000) -- (-0.000000,-0.000000)
(0.000000,0.000000) -- (-0.000000,0.000000)
(0.000000,0.000000) -- (-0.000000,0.000000)
(0.000000,0.000000) -- (-0.000000,0.000000)
(-0.000000,0.000000) -- (-0.000000,0.000000)
(-0.000000,0.000000) -- (0.000000,0.000000)
(-0.000000,0.000000) -- (0.000000,0.000000)
(-0.000000,0.000000) -- (0.000000,0.000000)
(-0.000000,0.000000) -- (0.000000,0.000000)
(-0.000000,0.000000) -- (0.000000,0.000000)
(-0.000000,0.000000) -- (0.000000,0.000000)
(-0.000000,0.000000) -- (0.000000,0.000000)
(-0.000000,0.000000) -- (0.000000,0.000000)
(-0.000000,0.000000) -- (0.000000,0.000000)
(-0.000000,0.000000) -- (0.000000,0.000000)
(-0.000000,0.000000) -- (0.000000,0.000000)
(-0.000000,-0.000000) -- (0.000000,0.000000)
(-0.000000,-0.000000) -- (0.000000,-0.000000)
(-0.000000,-0.000000) -- (0.000000,-0.000000)
(-0.000000,-0.000000) -- (0.000000,-0.000000)
(0.000000,-0.000000) -- (0.000000,-0.000000)
(0.000000,-0.000000) -- (-0.000000,-0.000000)
(0.000000,-0.000000) -- (-0.000000,-0.000000)
(0.000000,-0.000000) -- (-0.000000,-0.000000)
(0.000000,-0.000000) -- (-0.000000,-0.000000)
(0.000000,-0.000000) -- (-0.000000,-0.000000)
(0.000000,-0.000000) -- (-0.000000,-0.000000)
(0.000000,-0.000000) -- (-0.000000,-0.000000)
(0.000000,-0.000000) -- (-0.000000,-0.000000)
(0.000000,-0.000000) -- (-0.000000,-0.000000)
(0.000000,-0.000000) -- (-0.000000,-0.000000)
(0.000000,-0.000000) -- (-0.000000,-0.000000)
(0.000000,0.000000) -- (-0.000000,-0.000000)
(0.000000,0.000000) -- (-0.000000,0.000000)
(0.000000,0.000000) -- (-0.000000,0.000000)
(0.000000,0.000000) -- (-0.000000,0.000000)
(0.000000,0.000000) -- (-0.000000,0.000000)
(0.000000,0.000000) -- (-0.000000,0.000000)
(-0.000000,0.000000) -- (-0.000000,0.000000)
(-0.000000,0.000000) -- (0.000000,0.000000)
(-0.000000,0.000000) -- (0.000000,0.000000)
;
\draw (-0.000000,-0.400000) -- (0.000000,4.400000) node[right] {$\mathcal{C}_1$};
\draw (0.106667,-0.400000) -- (-1.173333,4.400000) node[right] {$\mathcal{C}_2$};
\end{tikzpicture}
		\subcaption{}\label{fig:schema_c}
	\end{minipage}%
	\caption{\label{alternate_projections} The trajectory of $\gamma^{\ell}$ given by the SK algorithm is illustrated for decreasing values of $\epsilon$ in (a) and (b). Overrelaxed projections (c) typically accelerate the convergence rate.}
\end{figure}

In what follows, we first define overrelaxed Bregman projections. We then propose a Lyapunov function that is used to show the global convergence of our proposed algorithm in section \ref{sec:algo}. The local convergence rate is then discussed in section \ref{section:local}.

\subsection{Overrelaxed projections}
We recall that $P_{\Ccal_k}$ are   operators  realizing the Bregman projection of  matrices $\gamma\in\IR^{n_1 n_2}$ onto the affine sets $\Ccal_k$, $k=1,2$, as  defined in \eqref{scaling}.
For $\omega\geq 0$, we now define the $\omega$-relaxed  projection operator $P^\omega_{\Ccal_k}$  as
\begin{equation}\label{eq:def_or_proj}
	\log P^\omega_{\Ccal_k}(\gamma) = (1-\omega) \log \gamma + \omega \log P_{\Ccal_k}(\gamma),
\end{equation}
where the logarithm is taken coordinate-wise.
Note that $P_{\Ccal_k}^0$ is the identity, $P_{\Ccal_k}^1 = P_{\Ccal_k}$ is the standard Bregman projection ) and $P_{\Ccal_k}^2$ is an involution (in particular because $\Ccal_k$ is an affine subspace). In the following, we will consider overrelaxations corresponding to $\omega\in[1;2)$
A naive algorithm would then consist in iteratively applying $P^\omega_{\Ccal_2}\circ P^\omega_{\Ccal_1}$ for some choice of $\omega$.
While it often behaves well in practice, this algorithm may sometimes not converge even for reasonable values of $\omega$.
Our goal in this section is to make this algorithm robust and to guarantee its global convergence.
\begin{remark}
Duality gives another point of view on the iterative overrelaxed Bregman projections: they indeed correspond to a successive overrelaxation (SOR) algorithm on the dual objective $E$ given in \eqref{DROT}. This is a procedure which, starting from $(\da^0,\db^0)=(\mathbf{0},\mathbf{0})$, defines for $\ell \in \mathbb{N}^*$,
\begin{align}\label{SORdual}
	\da^{\ell+1} &= (1-\omega)\da^{\ell} + \omega \argmax_\da E(\da,\db^{\ell})\\
	\db^{\ell + 1}&=(1-\omega) \db^{\ell} + \omega \argmax_\db E(\da^{\ell+1},\db),\label{SORdualb}
\end{align} 
From the definition of the projections in \eqref{scaling} and using again the relationships $\uu_i=e^{\da_i/\epsilon}$, $\vv_j=e^{\db_j/\epsilon}$ and $\gamma^0_{i,j}=e^{-c_{i,j}/\epsilon}$, expressions \eqref{SORdual} and \eqref{SORdualb} can be seen as overrelaxed projections \eqref{eq:def_or_proj}.
\end{remark}

\subsection{Lyapunov function}
Convergence of the successive overrelaxed projections is not guaranteed in general. In order to derive a robust algorithm with provable convergence, we introduce the Lyapunov function 
\begin{equation}\label{eq:lyapunov_function}
	F(\gamma) = \KL(\gamma^*, \gamma),
\end{equation}
where $\gamma^*$ denotes the solution of the regularized OT problem. We will use this function to enforce the strict descent criterion $F(\gamma^{\ell+1}) < F(\gamma^\ell)$ as long as the process has not converged.

The choice of \eqref{eq:lyapunov_function} as a Lyapunov function is of course related to the fact that Bregman projections are used throughout the algorithm. Further, we will show (Lemma \ref{lemma:lyapunov_decrease}) that its decrease is simple to compute and this descent criterion still allows enough freedom in the choice of the overrelaxation parameter.

Crucial properties of this Lyapunov function are gathered in the next lemma.
\begin{lemma} \label{lemma:lyapunov_decrease}
	For any $M \in \IR_+^*$, the sublevel set $\left\{ \gamma \mid F(\gamma) \le M \right\}$ is compact.
	Moreover, for any $\gamma$ in $\IR^{mn}_{+*}$, the decrease of the Lyapunov function after an overrelaxed projection can be computed as
	\begin{equation} \label{eq:kl_diff_scal}
		F(\gamma) - F(P^\omega_{\Ccal_k}(\gamma)) = 
		\scal{\mu_k}{\varphi_\omega \left((A_k \gamma) \oslash \mu_k \right)},
	\end{equation}
	where
	\begin{equation}
		\varphi_\omega(x) = x(1-x^{-\omega}) - \omega \log x
	\end{equation}
	is a real function, applied coordinate-wise.
\end{lemma}

\begin{proof}
	The fact that the Kullback-Leibler divergence is jointly lower semicontinuous implies in particular that $K$ is closed. Moreover, $K\subset \mathbb{R}^{n_1\times n_2}_+$ is bounded because $F$ is the sum of nonnegative, coercive functions of each component of its argument $\gamma$.
	
	Formula \eqref{eq:kl_diff_scal} comes from the expression $$F(\gamma^1)-F(\gamma^2)= \sum_{i,j}\left(\gamma^*_{i,j}\log(\gamma^2_{i,j}/\gamma^1_{i,j})+\gamma^1_{i,j}-\gamma^2_{i,j}\right)$$ and  relations \eqref{eq:def_or_proj} and \eqref{scaling}.
\end{proof}

be calculated without knowing $\gamma^*$, as shown by the following lemma.

It follows from Lemma \ref{lemma:lyapunov_decrease} that the decrease of $F$ for an overrelaxed projection is very cheap to estimate, since its computational cost is linear with respect to the dimension of data $\mu_k$. In Figure \ref{phi_omega}, we display the function  $\varphi_\omega(x)$. Notice that for the Sinkhorn--Knopp algorithm, which corresponds to $\omega=1$, the function $\varphi_\omega$ is always nonnegative. For other values $1\le\omega<2$, it is nonnegative for $x$ close to 1.

\begin{figure}[ht!]
	\begin{center}
		\includegraphics[width=8cm]{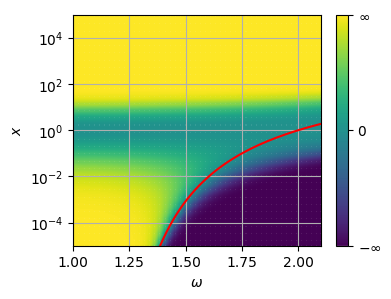}
		\caption{\label{phi_omega} Value of $\varphi_\omega(x)$. The function is positive above the red line, negative below. For any relaxation parameter $\omega$ smaller than $2$, there exists a neighborhood of $1$ on which $\varphi_\omega(\cdot)$ is positive.}
	\end{center}
\end{figure}

\subsection{Proposed algorithm}\label{sec:algo}
We first give a general convergence result that later serves as a basis to design an explicit algorithm.
\begin{theorem}\label{thm:algo}
	Let $\func_1$ and $\func_2$ be two continuous functions of $\gamma$ such that
	\begin{equation}\label{eq:cond_theta_k}
		\forall \gamma \in \IR_{+*}^{n_1 n_2},\quad
		F(P_{\Ccal_k}^{\func_k(\gamma)}(\gamma)) \le F(\gamma) ,
	\end{equation}
	where the inequality is strict whenever $\gamma \notin \Ccal_k$.
	Consider the sequence defined by $\gamma^0 = e^{-c/\epsilon}$ and
	\begin{align*}
		\tilde{\gamma}^{\ell+1} &= P_{\Ccal_1}^{\func_1(\gamma^{\ell})}(\gamma^{\ell}) \\
		\gamma^{\ell+1} &= P_{\Ccal_2}^{\func_2(\tilde{\gamma}^{\ell+1})}(\tilde{\gamma}^{\ell+1}).
	\end{align*}
	Then the sequence $(\gamma^{\ell})$ converges to $\gamma^*$.
\end{theorem}

\begin{lemma}[]
	\label{lemma:trivial_intersection}
	Let us take $\gamma^0$ in $\IR_{+*}^{n_1 n_2}$,
	and denote
	\[
	S = \left\{
	\diag(\uu) \gamma^0 \diag(\vv),\quad
	(\uu,\vv) \in \IR_{+*}^{n_1 + n_2}
	\right\}
	\]
	the set of matrices that are diagonally similar to $\gamma^0$.
	Then the set $S \cap \Ccal_1 \cap \Ccal_2$ contains exactly one element $\gamma^* = P_{\Ccal_1 \cap \Ccal_2}(\gamma^0)$.
\end{lemma}
\begin{proof}
	We refer to \cite{cuturi13} for a proof of this lemma.
\end{proof}

\begin{proof}[Proof of the theorem]
	First of all, notice that the operators $P_{\Ccal_k}^\theta$ apply a scaling to lines or columns of matrices. All $(\gamma^{\ell})$ are thus diagonally similar to $\gamma^0$:
	\[
	\forall \ell\ge0,\quad \gamma^{\ell} \in S
	\]
	By construction of the functions $\func_k$, the sequence of values of the Lyapunov function $(F(\gamma^{\ell}))$ is non-increasing. Hence $(\gamma^{\ell})$ is precompact.
	If $\zeta$ is a cluster point of $(\gamma^{\ell})$, let us define
	\begin{align*}
		\tilde{\zeta} &= P_{\Ccal_1}^{\func_1(\zeta)}(\zeta) \\
		\zeta' &= P_{\Ccal_2}^{\func_2(\tilde{\zeta})}(\tilde{\zeta}).
	\end{align*}
	Then by continuity of the applications, $F(\zeta) = F(\tilde{\zeta}) = F(\zeta')$. 
	From the hypothesis made on $\func_1$ and $\func_2$, it can be deduced that $\zeta$ is in $\Ccal_1$, and is thus a fixed point of $P_{\Ccal_1}$, while $\tilde{\zeta}$ is in $\Ccal_2$. Therefore $\zeta' = \tilde{\zeta} = \zeta$ is in the intersection $\Ccal_1 \cap \Ccal_2$.
	By Lemma \ref{lemma:trivial_intersection}, $\zeta = \gamma^*$, and the whole sequence $(\gamma^{\ell})$ converges to the solution.
\end{proof}

We can construct explicitly functions $\func_k$ as needed in Theorem~\ref{thm:algo} using the following lemma.
\begin{lemma}\label{lemma:F_P_theta}
	Let $1\le \omega < {{\theta}}$. Then, for any $\gamma \in \IR_{+*}^{n_1 n_2}$, one has
	\begin{equation}\label{eq:F_P_theta}
		F(P^\omega_{\Ccal_k}(\gamma)) \le F(P^{{\theta}} _{\Ccal_k}(\gamma)).
	\end{equation}
	Moreover, equality occurs if and only if $\gamma \in \Ccal_k$.
\end{lemma}
\begin{proof}
	Thanks to Lemma \ref{lemma:lyapunov_decrease}, one knows that
	\[
	F(P^\omega_{\Ccal_k}(\gamma)) - F(P^{{\theta}} _{\Ccal_k}(\gamma))
	= \scal{\mu_k}{(\varphi_{{\theta}} - \varphi_\omega) \left( \frac{A_k \gamma}{\mu_k} \right) } .
	\]
	The function that maps $t \in [1,\infty)$ to $\varphi_t(x)$ is non-increasing since
	$\partial_t \varphi_t(x) =  (x^{1-t} - 1)\log x.$
	For $x\neq 1$, it is even strictly decreasing.
	Thus inequality \eqref{eq:F_P_theta} is valid, with equality \emph{iff} $A_k \gamma = \mu_k$.
\end{proof}

We now argue that a good choice for the functions $\func_k$ may be constructed as follows. Pick a target parameter ${\theta_0} \in [1;2)$, that will act as an upper bound for the overrelaxation parameter $\omega$, and a small security distance $\delta>0$. Define the functions $\Theta^*$ and $\Theta$ as
\begin{align}
	\label{eq:Theta_opt}
	\Theta^*(w) &= \sup \left\{\omega \in [1;2]  \mid \varphi_\omega\left(\min w\right) \ge 0 \right\} ,\\
	\label{eq:Theta}
	\Theta(w) &= \min(\max(1,\Theta^*(w)-\delta),{\theta_0}),
\end{align}
where $\min w$ denotes the smallest coordinate of the vector $w$. 

\begin{proposition}\label{prop:thetachoice}
	The function
	\begin{equation}
		\label{eq:theta_k}
		\func_k(\gamma) =\Theta\left ((A_k \gamma)\oslash \mu_k\right)
	\end{equation}
	is continuous and verifies the descent condition (\ref{eq:cond_theta_k}).
\end{proposition}
\begin{proof}
	Looking at Figure~\ref{phi_omega} can help understand this proof.	Since the partial derivative of $\partial_\omega \varphi_\omega(x)$ is nonzero for any $x<1$, the implicit function theorem proves the continuity of $\Theta^*$.
	The function $\Theta^*\left((A_k \gamma)\oslash \mu_k)\right)$ is such that every term in relation (\ref{eq:kl_diff_scal}) is non-negative.
	Therefore, by Lemma \ref{lemma:F_P_theta}, using this parameter minus $\delta$ ensures the strong decrease (\ref{eq:cond_theta_k}) of the Lyapunov function.
	Constraining the parameter to $[1,{\theta_0}]$ preserves this property. 
\end{proof}

This construction, which is often an excellent choice in practice, has several advantages:
\begin{enumerate}[align=right,labelsep=10pt,leftmargin=0.5cm]
	\item[$\bullet$] it allows to choose arbitrarily the parameter ${\theta_0}$ that will be used eventually when the algorithm is close to convergence (we motivate what are good choices for ${\theta_0}$ in Section \ref{section:local});
	\item[$\bullet$] it is also an easy approach to having an adaptive method, as the approximation of $\Theta^*$ has a negligible cost (it only requires to solve a one dimensional problem that depends on the smallest value of $(A_k \gamma)\oslash \mu_k$, which can be done in a few iterations of Newton's method).
\end{enumerate}

The resulting algorithm, which is proved to be convergent by Theorem~\ref{thm:algo}, is written in pseudo-code in Algorithm \ref{SOR}. The implementation in the log domain is also given in Algorithm \ref{SOR_log}. Both processes use the function $\Theta$ defined implicitly in \eqref{eq:Theta}. The evaluation of  $\Theta$  is approximated in practice with a few iterations of Newton's method on the function $\omega \mapsto \varphi_\omega(\min u)$ which is decreasing as it can be seen on Figure~\ref{phi_omega}. With the choice ${\theta_0}=1$, one recovers exactly the original SK algorithm.
\begin{algorithm}
	\caption{Overrelaxed SK algorithm }
	\label{SOR}
	\begin{algorithmic}
		\REQUIRE $\mu_1\in \IR^{n_1}$, $\mu_2\in \IR^{n_2}$, $c\in \IR^{n_1\times n_2}_+$
		\STATE Set $\uu=\mathbf{1}_{n_1}$, $\vv=\mathbf{1}_{n_2}$, $\gamma^0=e^{-c/\epsilon}$, ${\theta_0}\in[1;2)$ and $\eta>0$
		\WHILE {$||\uu\otimes \gamma^0 \vv -  \mu_1||>\eta$}
		\STATE $\tilde \uu=\mu_1\oslash (\gamma^0 \vv)$, 
		\STATE $\omega=\Theta(\uu\oslash\tilde \uu)$
		\STATE  $\uu=\uu^{1-\omega}\otimes \tilde \uu^\omega$
		\STATE $\tilde \vv=\mu_2\oslash (^t\gamma^0  \uu)$
		\STATE $\omega=\Theta(\vv\oslash\tilde \vv)$
		\STATE  $\vv=\vv^{1-\omega}\otimes \tilde \vv^\omega$
		\ENDWHILE
		\RETURN $\gamma=\diag(\uu)\gamma^0\diag(\vv)$
	\end{algorithmic}
\end{algorithm}

\begin{algorithm}
	\caption{Overrelaxed SK algorithm in the log domain}
	\label{SOR_log}
	\begin{algorithmic}
		\REQUIRE $\mu_1\in \IR^{n_1}$, $\mu_2\in \IR^{n_2}$, $c\in \IR^{n_1\times n_2}_+$
		\STATE Set $\da=\mathbf{0}_{n_1}$, $\db=\mathbf{0}_{n_2}$, $\gamma^0=e^{-c/\epsilon}$, ${\theta_0}\in[1;2)$ and $\eta>0$
		\WHILE {$||\exp{(\da/\epsilon)}\otimes (\gamma^0 \exp{(\db/\epsilon)}) -  \mu_1||>\eta$}
		
		\STATE $r_i=\sum_j \exp{((-c_{i,j}+\da_i+\db_j)/\epsilon}-\log(\mu_1)_i$),\hfill $i=1\cdots n_1$
		\STATE $\tilde \da=\da-\epsilon \log r$
		\STATE $\omega=\Theta(r)$
		\STATE  $\da=(1-\omega)\da+\omega \tilde \da$

		\STATE $s_j=\sum_i \exp{((-c_{i,j}+\da_i+\db_j)/\epsilon}-\log(\mu_2)_j)$,\hfill $j=1\cdots n_2$
		\STATE $\tilde \db =\db-\epsilon \log s$
		\STATE $\omega=\Theta(s)$
		\STATE  $\db=(1-\omega)\db+\omega \tilde \db$
		\ENDWHILE
		\RETURN $\gamma=\diag(\exp(\da/\epsilon))\gamma^0\diag(\exp(\db/\epsilon))$
	\end{algorithmic}
\end{algorithm}

\subsection{Acceleration of local convergence rate}
\label{section:local}
In order to justify the acceleration of convergence that is observed in practice, we now study the local convergence rate of the overrelaxed algorithm, which follows from the classical convergence analysis of the linear SOR method. Our result involves the second largest eigenvalue of the matrix
\begin{equation}\label{eq:M1}
	M_1= \diag(1\oslash \mu_1) \, \gamma^* \, \diag(1\oslash \mu_2)\, ^t\gamma^*
\end{equation}
where $\gamma^*$ is the solution to the regularized OT problem (the largest eigenvalue is $1$, associated to the eigenvector $\mathbf{1}$). We denote the second largest eigenvalue by $1-\eta$, it satisfies $\eta>0$  \cite{knight2008sinkhorn}.

\begin{proposition}\label{prop:local}
	The SK algorithm converges locally at a linear rate $1-\eta$. For the optimal choice of extrapolation parameter $\theta^* = 2/(1+\sqrt{\eta})$, the overrelaxed projection algorithm converges locally linearly at a rate $(1-\sqrt{\eta})/(1+\sqrt{\eta})$. The local convergence of the overrelaxed algorithm is guaranteed for $\theta \in {]0,2[}$ and the linear rate is given on Figure~\ref{fig:eigtransform} as a function of $1-\eta$ and $\theta$.
\end{proposition}

\begin{proof}
	In this proof, we focus on the dual problem and we recall the relationship $\gamma^{\ell}=e^{\da^\ell/\epsilon}\gamma^0e^{\db^\ell/\epsilon}$ between the iterates of the overrelaxed projection algorithm $\gamma^\ell$ and the  iterates $(\da^\ell,\db^\ell)$ of the SOR algorithm on the dual problem~\eqref{SORdual}, initialized with $(\da^0,\db^0)=(0,0)$. 
	The dual problem~\eqref{DROT} is invariant by translations of the form $(\da,\db)\mapsto (\da-k,\db+k)$, $k\in \IR$, but is strictly convex up to this invariance. In order to deal with this invariance, consider the subspace $S$ of pairs of dual variables $(\da,\db)$ that satisfy $\sum \da=\sum \db$, let $\pi_S$ be the orthogonal projection on $S$ of kernel $(\mathbf{1},-\mathbf{1})$ and let $(\da^*,\db^*)\in S$ be the unique dual maximizer in $S$. 

	Since one SOR iteration is a smooth map, the local convergence properties of the SOR algorithm are characterized by the local convergence of its linearization, which here corresponds to the SOR method applied to the maximization of the quadratic Taylor expansion of the dual objective $E$ at $(\da^*,\db^*)$. This defines an affine map $M_\theta : (\da^{\ell},\db^{\ell})\mapsto (\da^{\ell+1},\db^{\ell+1})$ whose spectral properties are well known~\cite{ciarlet1982introduction, young2014iterative} (see also \cite[chapter 4]{chizat2017these} for the specific case of convex minimization and \cite{hadjidimos1985optimization} for the non-strictly convex case). For the case $\theta=1$, this is the matrix $M_1$ defined in Eq.~\eqref{eq:M1}. The operator norm of $\pi_S \circ M_1$ is smaller than $1-\eta$ so the operator $ (\pi_S \circ M_1)^\ell = \pi_S\circ M_1^\ell$ converges at the linear rate $1-\eta$ towards $0$ (observe that by construction, $\pi_S$ and $M_1$ are co-diagonalisable and thus commute): for any $(\alpha,\beta)\in \mathbb{R}^{n_1}\times \mathbb{R}^{n_2}$, it holds $\Vert \pi_S\circ M_1^\ell (\alpha -\alpha^*,\beta-\beta^*) \Vert_2 \leq \Vert \pi_S (\alpha -\alpha^*,\beta-\beta^*)\Vert_2 (1-\eta)^\ell$. More generally, the convergence of $\pi_S\circ M_\theta^\ell$ is guaranteed for $\theta \in ]0,2[$, with the linear rate 
	\begin{align}\label{eq:rate_SOR}
	f(\theta,\eta) = \begin{cases}
		\theta-1 & \text{if $\theta>\theta^*$}\\
		\frac12 \theta^2(1-\eta) - (\theta-1) +\frac12 \sqrt{(1-\eta)\theta^2(\theta^2(1-\eta)-4(\theta-1))} & \text{otherwise}.
	\end{cases}
	\end{align}
	This function is minimized with $\theta^*\coloneqq2/(1+\sqrt{\eta})$, which satisfies $f(\theta^*,\eta)=(1-\sqrt{\eta})/(1+\sqrt{\eta})$. The function $f$ is plotted in Figure \ref{fig:eigtransform}.
	
	To switch from these dual convergence results to primal convergence results, remark that $\gamma^\ell \to \gamma^*$ implies $\KL(\gamma^\ell,\gamma^0)\to \KL(\gamma^*,\gamma^0)$ which in turn implies $E(\da^\ell,\db^\ell) \to \max E$ so by invoking the partial strict concavity of $E$, we have $\pi_S(\da^\ell,\db^\ell)\to (\da^*,\db^*)$. The converse implication is direct so it holds $[\pi_S(\da^\ell,\db^\ell)\to (\da^*,\db^*)] \Leftrightarrow [\gamma^\ell \to \gamma^*]$. We conclude by noting the fact that $\pi_S(\da^\ell, \db^\ell)$ converges at a linear rate implies the same rate on $\gamma^\ell$, thanks to the relationship between the iterates.
\end{proof}

\begin{figure}
	\centering
	\includegraphics[scale=0.8]{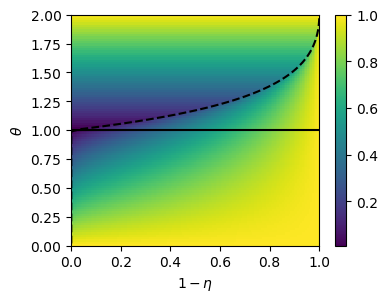}
	\caption{\label{fig:eigtransform} Local linear rate of convergence of the overrelaxed algorithm as a function of $1-\eta$, the local convergence rate of SK algorithm and $\theta$ the overrelaxation parameter. (plain curve) the original rate is recovered for $\theta=1$. (dashed curve) optimal overrelaxation parameter $\theta^*$.}
\end{figure}

\begin{corollary}
	The previous local convergence analysis applies to Algorithm \ref{SOR_log} with $\Theta$ defined as in Eq.~\eqref{eq:Theta} and the local convergence rate is given by the function of Eq.~\eqref{eq:rate_SOR} evaluated at the target extrapolation parameter ${\theta_0}$.
\end{corollary}
\begin{proof}
	What we need to show is that eventually one always has $\Theta(\gamma^\ell)={\theta_0}$. This can be seen from the quadratic Taylor expansion $\varphi_{{\theta_0}}(1+z)=z^2({\theta_0} -{\theta_0}^2/2)+ o(z^2)$, which shows that for any choice of ${\theta_0}\in {]1,2[}$, there is a neighborhood of $1$ on which $\varphi_{{\theta_0}}(\cdot)$ is nonnegative.
\end{proof}

\section{Experimental results}
We compare Algorithm~\ref{SOR} to SK algorithm on two very different optimal transport settings.
In setting (a) we consider the domain $[0,1]$ discretized into $100$ samples and the squared Euclidean transport cost on this domain.
The marginals are densities made of the sum of a base plateau of height $0.1$ and another plateau of height and boundaries chosen uniformly in $[0,1]$, subsequently normalized.
In setting (b) the cost is a $100\times 100$ random matrix with entries uniform in $[0,1]$ and the marginals are uniform.

Given an estimation of $1-\eta$, the local convergence rate of SK algorithm, we define ${\theta_0}$ as the optimal parameter as given in Proposition \ref{prop:local}. For estimating $\eta$, we follow two strategies. For strategy ``estimated'' (in blue on Figure~\ref{fig:comparison}), $\eta$ is measured by looking at the local convergence rate of SK run on another random problem of the same setting and for the same value of $\epsilon$. For strategy ``measured'' (in orange on Figure~\ref{fig:comparison}) the parameter is set using the local convergence rate of SK run on the same problem. Of course, the latter is an unrealistic strategy but it is interesting to see in our experiments that the ``estimated'' strategy performs almost as well as the ``measured'' one, as shown on~\ref{fig:comparison}.
\begin{figure}
	\begin{minipage}[b]{.5\textwidth}
		\centering
		\includegraphics[scale=0.575]{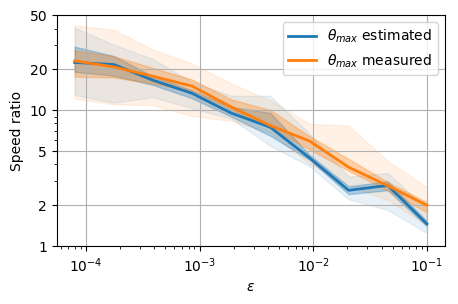}
		\subcaption{Quadratic cost, random marginals}\label{fig:compare_a}
	\end{minipage}%
	\begin{minipage}[b]{.5\textwidth}
		\centering
		\includegraphics[scale=0.575]{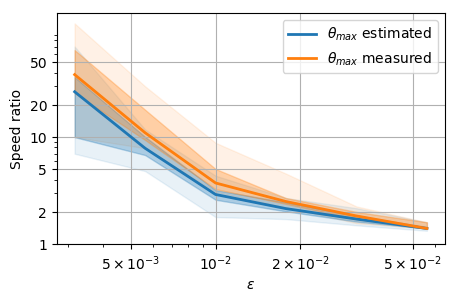}
		\subcaption{Random cost, uniform marginals}\label{fig:compare_b}
	\end{minipage}
	\caption{\label{fig:comparison}Speed ratio between SK algorithm and its accelerated SOR version Algorithm~\ref{SOR} w.r.t parameter $\epsilon$.}
\end{figure}

Figure~\ref{fig:comparison} displays the ratio of the number of iterations required to reach a precision of $10^{-6}$ on the dual variable $\da$ for SK algorithm and Algorithm \ref{SOR}. It is is worth noting that the complexity per iteration of these algorithms is the same modulo negligible terms, so this ratio is also the runtime ratio (our algorithm can also be parallelized on GPUs just as SK algorithm). In both experimental settings, for low values of the regularization parameter $\epsilon$, the acceleration ratio is above $20$  with Algorithm \ref{SOR}.

\section{Conclusion and perspectives}
The SK algorithm is widely used to solve entropy regularized OT.
In this paper we have first shown that RNA methods are adapted to the numerical acceleration of the SK algorithm. Nevertheless the global convergence of such approaches may not be ensured. 

Next, we demonstrate that the use of overrelaxed projections is a natural and simple idea to ensure and accelerate the convergence, while keeping many nice properties of the SK algorithm (first order, parallelizable, simple).
We have proposed an algorithm that adaptively chooses the overrelaxation parameter so as to guarantee global convergence.
The acceleration of the convergence speed is numerically impressive, in particular in low regularization regimes.
It is theoretically supported in the local regime by the standard analysis of SOR iterations.

This idea of overrelaxation can be generalized to solve more general problems such as multi-marginal OT, barycenters, gradient flows, unbalanced OT~\cite[chap. 4]{chizat2017these} but there is no systematic way to derive globally convergent algorithms.
Our work is a step in the direction of building and understanding the properties of robust first order algorithms for solving OT. More understanding is needed regarding SOR itself (global convergence speed, choice of ${\theta_0}$), but also its relation to other acceleration methods~\cite{2016arXiv160604133S,2017arXiv170509634A}.

\section*{Acknowledgments}
This study has been carried out with financial support from the French State, managed by the French National Research Agency (ANR) in the frame of the  GOTMI project (ANR-16-CE33-0010-01).

\bibliographystyle{abbrv}

\begin{thebibliography}{10}

\bibitem{alaya2019screening}
M.~Z. Alaya, M.~Berar, G.~Gasso, and A.~Rakotomamonjy.
\newblock Screening sinkhorn algorithm for regularized optimal transport.
\newblock {\em arXiv preprint arXiv:1906.08540}, 2019.

\bibitem{altschuler2017near}
J.~Altschuler, J.~Weed, and P.~Rigollet.
\newblock Near-linear time approximation algorithms for optimal transport via
  sinkhorn iteration.
\newblock {\em arXiv preprint arXiv:1705.09634}, 2017.

\bibitem{2017arXiv170509634A}
J.~{Altschuler}, J.~{Weed}, and P.~{Rigollet}.
\newblock {Near-linear time approximation algorithms for optimal transport via
  Sinkhorn iteration}.
\newblock {\em ArXiv e-prints}, arXiv:1705.09634, 2017.

\bibitem{anderson1965iterative}
D.~G. Anderson.
\newblock Iterative procedures for nonlinear integral equations.
\newblock {\em Journal of the ACM (JACM)}, 12(4):547--560, 1965.

\bibitem{benamou15}
J.-D. Benamou, G.~Carlier, M.~Cuturi, L.~Nenna, and G.~Peyr{\'e}.
\newblock Iterative {Bregman} projections for regularized transportation
  problems.
\newblock {\em SIAM Journal on Scientific Computing}, 37(2):A1111--A1138, 2015.

\bibitem{bregman67}
L.~M. Bregman.
\newblock The relaxation method of finding the common point of convex sets and
  its application to the solution of problems in convex programming.
\newblock {\em USSR computational mathematics and mathematical physics},
  7(3):200--217, 1967.

\bibitem{chizat2017these}
L.~Chizat.
\newblock {\em Unbalanced optimal transport: models, numerical methods,
  applications}.
\newblock PhD thesis, Universit{\'e} Paris Dauphine, 2017.

\bibitem{2016arXiv160705816C}
L.~{Chizat}, G.~{Peyr{\'e}}, B.~{Schmitzer}, and F.-X. {Vialard}.
\newblock {Scaling Algorithms for Unbalanced Transport Problems}.
\newblock {\em ArXiv e-prints}, arXiv:1607.05816, 2016.

\bibitem{ciarlet1982introduction}
P.~Ciarlet.
\newblock {\em Introduction {\`a} l'analyse num{\'e}rique matricielle et {\`a}
  l'optimisation}.
\newblock masson, 1982.

\bibitem{2015arXiv150700504C}
N.~{Courty}, R.~{Flamary}, D.~{Tuia}, and A.~{Rakotomamonjy}.
\newblock {Optimal Transport for Domain Adaptation}.
\newblock {\em ArXiv e-prints}, arXiv:1507.00504, 2015.

\bibitem{cuturi13}
M.~Cuturi.
\newblock Sinkhorn distances: Lightspeed computation of optimal transport.
\newblock In {\em Advances in Neural Information Processing Systems (NIPS'13)},
  pages 2292--2300, 2013.

\bibitem{dessein2016}
A.~Dessein, N.~Papadakis, and J.-L. Rouas.
\newblock Regularized optimal transport and the rot mover's distance.
\newblock {\em The Journal of Machine Learning Research}, 19(1):590--642, 2018.

\bibitem{2015arXiv150605439F}
C.~{Frogner}, C.~{Zhang}, H.~{Mobahi}, M.~{Araya-Polo}, and T.~{Poggio}.
\newblock {Learning with a Wasserstein Loss}.
\newblock {\em ArXiv e-prints}, arXiv:1506.05439, 2015.

\bibitem{2014arXiv1412.7457G}
E.~{Ghadimi}, H.~R. {Feyzmahdavian}, and M.~{Johansson}.
\newblock {Global convergence of the Heavy-ball method for convex
  optimization}.
\newblock {\em ArXiv e-prints}, arXiv:1412.7457, 2014.

\bibitem{hadjidimos1985optimization}
A.~Hadjidimos.
\newblock On the optimization of the classical iterative schemes for the
  solution of complex singular linear systems.
\newblock {\em SIAM Journal on Algebraic Discrete Methods}, 6(4):555--566,
  1985.

\bibitem{iutzeler2019generic}
F.~Iutzeler and J.~M. Hendrickx.
\newblock A generic online acceleration scheme for optimization algorithms via
  relaxation and inertia.
\newblock {\em Optimization Methods and Software}, 34(2):383--405, 2019.

\bibitem{knight2008sinkhorn}
P.~A. Knight.
\newblock The {S}inkhorn--{K}nopp algorithm: convergence and applications.
\newblock {\em SIAM Journal on Matrix Analysis and Applications},
  30(1):261--275, 2008.

\bibitem{lehmann2020note}
T.~Lehmann, M.-K. von Renesse, A.~Sambale, and A.~Uschmajew.
\newblock A note on overrelaxation in the sinkhorn algorithm.
\newblock {\em arXiv preprint arXiv:2012.12562}, 2020.

\bibitem{NIPS2016_6248}
G.~Montavon, K.-R. M\"{u}ller, and M.~Cuturi.
\newblock Wasserstein training of restricted boltzmann machines.
\newblock In D.~D. Lee, M.~Sugiyama, U.~V. Luxburg, I.~Guyon, and R.~Garnett,
  editors, {\em Advances in Neural Information Processing Systems 29}, pages
  3718--3726, 2016.

\bibitem{2016arXiv160609070O}
P.~{Ochs}.
\newblock {Local Convergence of the Heavy-ball Method and iPiano for Non-convex
  Optimization}.
\newblock {\em ArXiv e-prints}, arXiv:1606.09070, 2016.

\bibitem{peyre2016quantum}
G.~Peyr{\'e}, L.~Chizat, F.-X. Vialard, and J.~Solomon.
\newblock Quantum optimal transport for tensor field processing.
\newblock {\em ArXiv e-prints}, arXiv:1612.08731, 2016.

\bibitem{CuturiPeyre19}
G.~Peyr{\'e} and M.~Cuturi.
\newblock Computational optimal transport.
\newblock {\em Foundations and Trends{\textregistered} in Machine Learning},
  11(5-6):355--607, 2019.

\bibitem{POLYAK19641}
B.~Polyak.
\newblock Some methods of speeding up the convergence of iteration methods.
\newblock {\em USSR Computational Mathematics and Mathematical Physics}, 4(5):1
  -- 17, 1964.

\bibitem{Rabin2014}
J.~Rabin and N.~Papadakis.
\newblock Non-convex relaxation of optimal transport for color transfer between
  images.
\newblock In {\em NIPS Workshop on Optimal Transport for Machine Learning
  (OTML'14)}, 2014.

\bibitem{richardson1911ix}
L.~F. Richardson.
\newblock Ix. the approximate arithmetical solution by finite differences of
  physical problems involving differential equations, with an application to
  the stresses in a masonry dam.
\newblock {\em Philosophical Transactions of the Royal Society of London.
  Series A, Containing Papers of a Mathematical or Physical Character},
  210(459-470):307--357, 1911.

\bibitem{Rolet2016}
A.~Rolet, M.~Cuturi, and G.~Peyr{\'{e}}.
\newblock Fast dictionary learning with a smoothed {Wasserstein} loss.
\newblock In {\em International Conference on Artificial Intelligence and
  Statistics (AISTATS)}, volume~51 of {\em JMLR Workshop and Conference
  Proceedings}, pages 630--638, 2016.

\bibitem{Rubner2000}
Y.~Rubner, C.~Tomasi, and L.~J. Guibas.
\newblock The {Earth Mover's Distance} as a metric for image retrieval.
\newblock {\em International Journal of Computer Vision}, 40(2):99--121, 2000.

\bibitem{2017arXiv170801955S}
M.~A. Schmitz, M.~Heitz, N.~Bonneel, F.~Ngole, D.~Coeurjolly, M.~Cuturi,
  G.~Peyr{\'e}, and J.-L. Starck.
\newblock Wasserstein dictionary learning: Optimal transport-based unsupervised
  nonlinear dictionary learning.
\newblock {\em SIAM Journal on Imaging Sciences}, 11(1):643--678, 2018.

\bibitem{schmitzer2016stabilized}
B.~Schmitzer.
\newblock Stabilized sparse scaling algorithms for entropy regularized
  transport problems.
\newblock {\em arXiv preprint arXiv:1610.06519}, 2016.

\bibitem{2016arXiv160604133S}
D.~{Scieur}, A.~{d'Aspremont}, and F.~{Bach}.
\newblock {Regularized Nonlinear Acceleration}.
\newblock {\em ArXiv e-prints}, arXiv:1606.04133, 2016.

\bibitem{scieur2018online}
D.~Scieur, E.~Oyallon, A.~d'Aspremont, and F.~Bach.
\newblock Online regularized nonlinear acceleration.
\newblock {\em arXiv preprint arXiv:1805.09639}, 2018.

\bibitem{seguy2015principal}
V.~Seguy and M.~Cuturi.
\newblock Principal geodesic analysis for probability measures under the
  optimal transport metric.
\newblock In {\em Advances in Neural Information Processing Systems}, pages
  3312--3320, 2015.

\bibitem{sinkhorn64}
R.~Sinkhorn.
\newblock A relationship between arbitrary positive matrices and doubly
  stochastic matrices.
\newblock {\em The annals of mathematical statistics}, 35(2):876--879, 1964.

\bibitem{Solomon2015}
J.~Solomon, F.~de~Goes, G.~Peyr{\'{e}}, M.~Cuturi, A.~Butscher, A.~Nguyen,
  T.~Du, and L.~Guibas.
\newblock Convolutional {Wasserstein} distances: Efficient optimal
  transportation on geometric domains.
\newblock {\em ACM Transactions on Graphics}, 34(4):66:1--66:11, 2015.

\bibitem{thibault2017overrelaxed}
A.~Thibault, L.~Chizat, C.~Dossal, and N.~Papadakis.
\newblock Overrelaxed sinkhorn-knopp algorithm for regularized optimal
  transport.
\newblock {\em arXiv preprint arXiv:1711.01851}, 2017.

\bibitem{young2014iterative}
D.~M. Young.
\newblock {\em Iterative solution of large linear systems}.
\newblock Elsevier, 2014.

\bibitem{Zavriev1993}
S.~K. Zavriev and F.~V. Kostyuk.
\newblock Heavy-ball method in nonconvex optimization problems.
\newblock {\em Computational Mathematics and Modeling}, 4(4):336--341, 1993.

\end{thebibliography}

\end{document}